\documentclass[%% 
prl%
 ,tightenlines%
,amssymb, nobibnotes, letterpaper, 11pt ]{revtex4-1}
\usepackage{docs}%
\usepackage{bm}%
\newcommand{\mytitle}{On exterior calculus and curvature in
  piecewise-flat manifolds}

\usepackage{graphicx}

\usepackage{amsthm, amsmath, amscd}

\usepackage[colorlinks=true,linkcolor=blue, urlcolor=red]{hyperref}%
\expandafter\ifx\csname package@font\endcsname\relax\else
 \expandafter\expandafter
 \expandafter\usepackage
 \expandafter\expandafter
 \expandafter{\csname package@font\endcsname}%
\fi
\hypersetup{
  pdftitle=\mytitle,
  pdfauthor= {Jonathan R. McDonald},
  pdfsubject={Differential Geometry},
  pdfkeywords = {Regge calculus, curvature, discrete differential
    geometry}
}

\usepackage{graphicx}
\usepackage{latexsym}
\usepackage{amsthm}
\usepackage{amscd,amssymb}
\usepackage{float}
\usepackage{mathrsfs}
\usepackage{dcolumn}
\usepackage{bm}
\usepackage{dsfont}

\newcommand{\defect}[1]{\ensuremath{\varepsilon_{#1}}}

 \newcommand{\ls}{\ensuremath{\ell}}
\newcommand{\vs}{\ensuremath{v}}
\newcommand{\tris}{\ensuremath{t}}
\newcommand{\tets}{\ensuremath{T}}
\newcommand{\polys}[1]{\ensuremath{s^{(#1)}}}
\newcommand{\ld}{\ensuremath{\lambda}}     %%%  dual edge
\newcommand{\vd}{\ensuremath{\nu}}            %%%  dual vertex
\newcommand{\polyd}[1]{\sigma^{(#1)}}               %%%  dual polytope
\newcommand{\hinge}{\ensuremath{h}}
    
\newcommand{\volav}[2]{ \ensuremath{\widetilde{\left\langle  #1
      \right\rangle_{#2}   }}}
\newcommand{\areaav}[2]{ \ensuremath{{\left\langle  #1 \right\rangle_{#2}   }}}

\newcommand{\circum}[1]{\ensuremath{{\cal C}_{#1}}}
\newcommand{\irrhyb}[1]{\ensuremath{V_{#1}}}

\newcommand{\Mani}{\ensuremath{{\cal M}}}
\newcommand{\TriMan}{\ensuremath{{\cal T}}}
\newcommand{\PFtri}{\ensuremath{{\cal T}_{0}}}

\newcommand{\latbound}{\ensuremath{\tilde{\delta}}}
\newcommand{\latcobound}{\ensuremath{\tilde{d}}}
\newcommand{\LIP}[2]{\ensuremath{        \left( #1 , #2   \right)              }}
\newcommand{\TSIP}[2]{\ensuremath{        \left\langle #1 | #2   \right\rangle              }}

\newcommand{\sovVec}[1]{\ensuremath{\tilde{e}_{#1}}}
\newcommand{\sovForm}[1]{\ensuremath{\tilde{e}^{#1}}}
\newcommand{\TSVec}[1]{\ensuremath{{e}_{#1}}}
\newcommand{\TSForm}[1]{\ensuremath{{e}^{#1}}}

\newcommand{\MomArm}[2]{\ensuremath{M_{#1 #2}}}
\newcommand{\Solder}[1]{\ensuremath{m_{#1}}}

\newcommand{\refEQ}[1]{Eq.~(\ref{#1})}
\newcommand{\refFig}[1]{Figure~\ref{#1}}
\newcommand{\refSec}[1]{Sec.~\ref{#1}}

\newtheorem{defines}{Definition}
\newtheorem{thm}{Theorem}
\newtheorem{cor}{Corollary}

\newcommand{\comment}[1]{}

\begin{document}

\title{\mytitle}%

\author{ Jonathan R. McDonald}	%
\email{jmcdnld@gmail.com}
\affiliation{Department of Mathematics, Harvard University, Cambridge, Massachusetts, 02138}
\affiliation{Air Force Research Laboratory, Information Directorate,
  Rome, New York, 13441}
\author{Warner A. Miller}
%\email{wam@fau.edu}
\affiliation{Department of Mathematics, Harvard University, Cambridge, Massachusetts, 02138}
\affiliation{Department of Physics, Florida Atlantic University, Boca Raton, Florida, 33431}
\author{ Paul M. Alsing}	%
%\email{paul.alsing@rl.af.mil}
\affiliation{Air Force Research Laboratory, Information Directorate,
  Rome, New York, 13441}
\author{Xianfeng David Gu}
%\email{gu@cs.sunysb.edu}
\affiliation{Department of Computer Science, Stony Brook University, Stony Brook, New York 11794}
\author{Xuping Wang}
%\email{xwang14@fau.edu}
\affiliation{Department of Physics, Florida Atlantic University, Boca
 Raton, Florida 33431}
\author{Shing-Tung Yau}
%\email{yau@math.harvard.edu}
\affiliation{Department of Mathematics, Harvard University, Cambridge, Massachusetts, 02138}

\date{\today}%
\begin{abstract}
  Simplicial, piecewise-flat discretizations of manifolds provide a
  clear path towards curvature analysis on discrete geometries and for
  solutions of PDE's on manifolds of complex topologies.  In this
  manuscript we review and expand on discrete exterior calculus
  methods using hybrid domains.  We then analyze the geometric
  structure of curvature operators in a piecewise-flat lattice.
\end{abstract} 
\maketitle
%\tableofcontents

\section{Introduction}\label{sec:intro}

%%%%%% PF manifolds in RC are often given by continuous approximations
%%%%%% to the discrete structure, DEC approaches give us lattice
%%%%%% element differential forms, here we propose a fully discrete
%%%%%% approach that is based on a Volume DEC (VDEC) approach to PF
%%%%%% manifolds.  This is helpful when we consider the Hodge duality,
%%%%%% Riemann curvature, and analysis of curvature forms in the
%%%%%% manifold.  It maps fully into a discrete view of the manifold
%%%%%% that smears out the distributional quantities to the resolution
%%%%%% of the lattice and assigns domains of support to the functions.

Piecewise-flat (PF) simplicial manifolds are a crucial computational
framework for systems with dynamic geometry or systems with complex
topologies/geometries.  In general relativity, PF manifolds are a
cornerstone of the coordinate-free discretization introduced by Regge
\cite{Regge:1961} often referred to as Regge calculus (RC).  RC is
often regarded as the backbone of the semi-classical or low-energy limit of,
or even an effective theory \cite{Williams:1992} of,  quantum gravity.
It has also proven to be a useful computational tool in numerical
relativity \cite{Gentle:2002}.  For numerical solutions to partial
differential equations, PF manifolds are one clear method to
discretize complex geometries over which the differential equations
act \cite{Arnold:2006} or in conformal transformations on 2
dimensional surfaces of arbitrary genus \cite{Gu:2003}.

One may view RC as an approach to characterizing the intrinsic
geometry of PF manifolds.  In other regards the exterior calculus
prescribes the calculus of fields on the curved background of smooth
manifolds.  Exterior calculus approaches to functions and fields on
and the curvature tensors of the PF manifold have been of great use in
preserving geometric notions on the discrete manifolds.  In the
canonical approach to RC, one can describe an exact action principle
\cite{FL:1984, Miller:HARC}, yet the discrete curvature tensors give
only approximate expressions after smoothing over the discontinuities.
The exterior calculus methods were also quite useful in developing a
clearer understanding of the Einstein tensor in RC \cite{Miller:1986,
  McDonald:BBP}.  In the numerical analysis for PDE's on PF manifolds,
discrete exterior calculus (DEC) is used to preserve geometric
symmetries of continuous systems on discrete manifolds
\cite{Desbrun:DEC} and for providing the bedrock on which to construct
geometric flows on complex topologies. \cite{Gu:2002, Gu:2003} The
latter has been extensively studied on 2D
surfaces.\cite{Gu:CompConfGeom, Gu:2004, Gu:2005a, Gu:2005b, Gu:2008}

In recent work we have extended and applied the methods of RC and DEC
to derive a simplicial discretization of Hamilton's Ricci flow
\cite{McDonald:RF}.  As we apply RC and discrete forms to dynamic, but
not necessarily covariant, geometric flows on PF manifolds, it is
necessary to develop a deeper understanding of the nature of curvature
and exterior calculus in discrete geometries.  In this manuscript we
show how the curvature in PF manifolds gives rise to seemingly
distinct notions of curvature tensors: (1) curvature with a single
sectional curvature or (2) isotropic curvatures similar to those of
Einstein spaces.  \refSec{subsec:canRC} and \refSec{subsec:DEC} will
review canonical RC and the standard approaches to discrete
differential forms.  Then in \refSec{sec:LocalForms} we discuss the
geometric principles behind hybrid cells as local measures and show
how these hybrid volumes are core elements of a volume-based DEC.  In
\refSec{sec:CurvPF} we discuss the representation of curvature
operators over the hybrid measures and transformations between them.

\subsection{Canonical Regge Calculus}\label{subsec:canRC}
%%%% In this section we build up the theory of RC based on the model
%%%% that functions on the manifold are analyzed in the piecewise-flat
%%%% domains where a clear interpretation of flat space is prevalent.
%%%% This follows Sorkin, TD Lee, Williams/Rocek.

Suppose $\Mani$ is a $d$-dimensional, smooth manifold endowed with a
simplicial complex $\TriMan$.  A PF triangulation of
$\Mani$, $\PFtri$, is a mapping from each $d$-simplex to a flat
$d$-simplex in $\mathds{R}^{d}$ such that the proper lengths of the
edges in the $1$-skeleton of $\TriMan$ are preserved in $\PFtri$.  The
simplicial manifold formed by $\PFtri$ is often called a Regge
manifold or Regge skeleton.  We now review the canonical approach to
RC,  as in  \cite{Sorkin:Evol, FL:1984}.

The interior geometry of any given simplex in $\PFtri$ is given by
Euclidean or Minkowski geometry and represents a common tangent
space for each of the vertexes of the simplex.  Any such simplex has
an induced metric that is uniquely determined by the proper
squared-edge lengths of the simplex.  The metric as a function of the
edge lengths $g_{\mu\nu}(\ell^{2})$ gives the local,
piecewise-constant approximation to the metric associated with
$\Mani$.  For two simplexes sharing a common boundary, the joint
domain is isomorphic to $\mathds{R}^{d}$ and thus is intrinsically
  flat.  While the two simplexes as viewed from an observer off the
  manifold may appear curved, i.e. with non-zero extrinsic curvature,
  a loop of parallel transport from simplex $A$ to simplex $B$ and
  back induces no change in orientation on a tangent vector in general
  position.  Formally this is related to the requirement of the
  existence of a metric compatible connection in the PF manifold.

While there exists a flat connection across any $(d-1)$--boundary,
curvature naturally arises when generating a map from a complete
set of $d$-simplexes sharing a common $(d-2)$--simplex to
$\mathds{R}^{d}$.  This is the first indication of curvature in the PF
manifold.  To effectively handle the discontinuity in mapping the
neighborhood of a $(d-2)$-simplex, or codimension-$2$ hinge $h$, we
can map the neighborhood of $h$ to a subspace of $\mathds{R}^{d}$ and
smoothly continue the mapping across the removed section of
$\mathds{R}^{d}$ resulting from `breaking' a $(d-1)$-simplex into two.
By breaking a $(d-1)$-simplex into two to make the mapping to
$\mathds{R}^{d}$, a defect in the correspondence between the interior
angles of the simplexes at $h$ and $2\pi$ is evident
(\refFig{Fig:defect}).  We take the deviation of the interior
angles of each simplex on $h$ from an exact embedding in flat space to
be the defect angle $\defect{h}$ associated to $h$ ;
\begin{equation}\label{eq:defect}
  \defect{h}   =   2\pi - \sum_{i|_{h}} \theta_{i}.
\end{equation}
This defect angle is the measure of curvature associated with parallel
transport of a vector around a loop that encircles $h$.  Indeed, if we
take any vector with components in the plane $\hat{h}^{*}$ orthogonal
to $h$ and transport it around the boundary of any area
$\sigma^{\alpha\beta}$ (also with components in $\hat{h}^{*}$), then
the vector will have rotated by an amount equal to $\defect{h}$, and
the rotation occurs in the plane of $\hat{h}^{*}$.  Any loop
$\sigma^{\alpha\beta}$ with components in $\hat{h}^{*}$ will generate
such rotations, independent of the area enclosed.  This has
implications for the sectional curvature associated with $h$.  The
sectional curvature associated with the loop $\sigma^{\alpha\beta}$ is
given by
\begin{equation}\label{eq:SecCurv}
  K_{\alpha\beta} = \frac{(\text{Angle of Rot'n})}{(\text{Area Enclosed})} = 
  \frac{\defect{h}}{|\sigma^{\alpha\beta}|}.
\end{equation}
Since the angle of rotation is independent of the area enclosed, we
can take an infinitesimal area of rotation encircling $h$ whose limit,
$|\sigma^{\alpha\beta}|\rightarrow 0$, impling a singularity in the
sectional curvature.  The singularity is known as a conic singularity
associated to the hinge $\hinge$.

\begin{center}
\begin{figure}[t]
\includegraphics[width=1.5in]{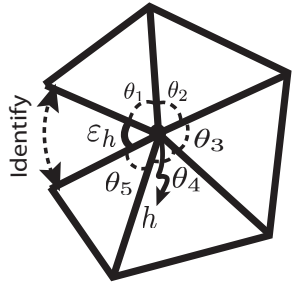}
\caption{$2$-dimensional projection of a mapping of the neighborhood
  of a hinge $h$ to $\mathds{R}^{d}$ with defect
  $\defect{h}$.}\label{Fig:defect}
\end{figure}
\end{center}

%%% Correspondence between the above notion and the Riemann tensor is
%%% difficult, but can be suitable achieved by integrating over a loop
%%% and using the above.

We can provide a representation for the sectional curvature for a loop
of parallel transport, but can we also provide a clear expression for
the Riemann curvature?  From the continuum theory with infinitesimal
rotations we can express the rotation of a vector $A^{\mu}$ after
transport along the boundary of an infinitesimal area by
\begin{equation}\label{eq:ContVecTrans}
\delta A^{\nu} = -R^{\nu}_{\ \ \mu \alpha\beta} A^{\mu} {\bf d}\sigma^{\alpha\beta}.
\end{equation} 
However, the curvature in PF manifolds is characterized by finite
rotations.  Friedberg and Lee \cite{FL:1984} showed that a
correspondence between the above notions of defect angle and a Riemann
tensor can be approximately given by
\begin{subequations}
\begin{equation} \label{eq:LeeRiem}
R_{12\ 12}(h) \approx \defect{h} \delta(x_1) \delta(x_{2})  \approx -R_{12 \ 21}(h)
\end{equation}
\begin{equation} \label{eq:LeeRiem1}
R_{1\ 1}(h) \approx \defect{h} \delta(x_1) \delta(x_{2})  \approx R_{2 \ 2}(h)
\end{equation}
\begin{equation} \label{eq:LeeScalCurv}
R(h) \approx 2\defect{h} \delta(x_1) \delta(x_{2}) 
\end{equation}
\begin{equation} \label{eq:LeeLagrang}
\sqrt{|g|}R(h) = 2\defect{h} \delta(x_1) \delta(x_{2}),
\end{equation}
\end{subequations}
where Eq.~(\ref{eq:LeeLagrang}) is exact and the coordinates $\{x_1,
x_2\}$ lie in the plane of $\hat{h}^*$.  In this approach, where one
 smooths out the discontinuities, one generates a limiting
sequence of surfaces approximating the plane $\hat{h}^{*}$ and takes
the limit to the PF surface.  This generates Dirac delta distributions
such that the curvature is evaluated only on the hinge $h$ and zero
elsewhere.  While the curvature tensors are only approximate in this
sense, the integrand of the Einstein-Hilbert action,
$$I_{\text{E-H}} = \frac{c^{4}}{16\pi G}\int \sqrt{|g|} R d^{d}x$$
is exact.  Thus we can take the standard action principle and have an
exact expression for the lattice geometry (first locally),
\begin{equation}\label{eq:REHaction}
\frac{c^{4}}{16\pi G}\int_{V_{h}} \sqrt{|g|} R d^{d}x = \frac{c^4}{8\pi G}\defect{h} A_{h}
\end{equation} 
where the integration is over a domain containing a single $h$.  The
global expression is obtained by summing over all hinges,
\begin{equation}\label{eq:REHaction1}
\frac{c^{4}}{16\pi G}\int \sqrt{|g|} R d^{d}x = \frac{c^4}{8\pi G}\sum_{h\in \PFtri}\defect{h} A_{h}.
\end{equation}

%%%%%%%%%%%%%%%%%%%%%%%%%%%%%%%%%%%%%%%%%%%%%%%%%%%%
%%%%%%%%%%%%%%%%%%%%%%%%%%%%%%%%%%%%%%%%%%%%%%%%%%%%

\subsection{Discrete Exterior Calculus}\label{subsec:DEC}
%%%% Here we develop the analysis of functions based on DEC/Comp EM
%%%% where field variables are assigned as integrated variables over
%%%% lattice k-elements.  This is a progression to our interpretation
%%%% of field variables defined 'locally' over hybrid domains.  This
%%%% follows Whitney, Gu, Bossavit, Desbrun/Hirani/Stern.

We now move away from the explicit geometry of PF manifolds to the
properties of differential forms on simplicial manifolds.  The
framework we will follow below has stemmed from the work of Whitney
\cite{Whitney:book} and, later, Bossavit \cite{Bossavit:1991}.  Recent
use of differential forms in PF manifolds corresponds to surface
parameterization \cite{Gu:2002} or finite-element methods
\cite{Arnold:2006, Desbrun:DEC}.  The use of differential forms in the
simplicial lattice has been independently used in RC (see for example
\cite{Sorkin:EM, Warner:1982, Barrett:1987, Miller:1986}).  In the
description below, we will follow some of the notation and conventions
of \cite{Desbrun:DEC}.

The algebraic structure of the simplicial lattice is such that we have
a natural representation of a discrete chain complex,
\begin{equation} \label{eq:chaincomp}
0 \stackrel{\latbound}{\rightarrow } \polys{d} \stackrel{\latbound}{\rightarrow} \polys{d-1}\stackrel{\latbound}{\rightarrow} \cdots\stackrel{\latbound}{\rightarrow} \polys{1} \stackrel{\latbound}{\rightarrow} \polys{0} \stackrel{\latbound}{\rightarrow} 0,
\end{equation}
and a discrete co-chain complex,
\begin{equation} \label{eq:cochaincomp}
0 \stackrel{\latcobound}{\rightarrow } \polys{0} \stackrel{\latcobound}{\rightarrow} \polys{1}\stackrel{\latcobound}{\rightarrow} \cdots\stackrel{\latcobound}{\rightarrow} \polys{d-1} \stackrel{\latcobound}{\rightarrow} \polys{d} \stackrel{\latcobound}{\rightarrow} 0.
\end{equation}
Moreover, given a geometric dual lattice to the simplicial skeleton one can
construct the dual chain and co-chain complexes.   

Given a $p$-form $\omega$ on the smooth manifold $\Mani$, we seek a
representation of $\omega$ on the PF manifold $\PFtri$.  The approach
taken in DEC discretizations is to take the smooth image of $\PFtri$
on $\Mani$, i.e. $\TriMan$, and evaluate $\omega$ on the $p$-skeleton
of $\TriMan$.  The simplicial approximation to $\omega$ is obtained by
integrating over individual elements of the $p$-skeleton of $\TriMan$
and associating that value with the image of the element in $\PFtri$;
\begin{equation}\label{eq:SimpForm}
\TSIP{\omega}{\polys{p}}:=\omega_{0}(\polys{p})  = \int_{\polys{p}}  \omega.
\end{equation}
Here we use the notation that $\polys{p}$ denotes an element of the
$p$-skeleton of $\TriMan$.  In general, we will use $\polys{p}$ to
denote both elements of $\TriMan$ and $\PFtri$. In the standard RC
approach to discrete differential forms, the discrete forms may often
be denoted by subscripts, e.g. $\omega_{\ell} = \frac{1}{|\ell|}
\TSIP{\omega}{\ell}$ will generally denote a one-form $\omega$
associated with an edge $\ell$ in the simplicial lattice.  Only when
we are considering the initial discretization will $\polys{p}$ be in
$\TriMan$.  Once the discrete differential form is assigned, all
further manipulations take place within $\PFtri$.

In addition to simplicial forms in $\PFtri$, one can also discretize
$\omega$ on the dual lattice.  Ambiguity arises here due to the
multitude of ways of assigning a dual lattice.  In general,
there exist several natural geometric dual lattices, e.g. barycentric,
circumcentric or incentric dual lattices, but there are also numerous
non-intuitive dualities that can be constructed for an arbitrary
simplicial lattice \cite{Richter:Diss}.  For our purposes and for
clarity in later sections, we take the circumcentric dual as our
prescribed dual lattice.  The circumcentric dual of a $k$-simplex is
given by \cite{Hirani:DEC}
\begin{align}
\star\polys{k} := \sum_{\polys{p}\ni \polys{k}:\ p>k} \text{sgn} \left[ \circum{k}, \circum{k+1}, \ldots, \circum{d}\right], 
\end{align}
where $\circum{p}$ is the circumcenter of a $\polys{p}$ and
$\text{sgn}$ is used to ensure an orientation consistent with the
encompassing $\polys{d}$. The circumcentric dual has several nice
properties that make it a natural choice: (1) a simplicial element and
its circumcentric dual are (locally) orthogonal to one another, (2)
the $p$-dimensional dual to a $p$-element is equidistant from each
vertex on the $p$-element, and (3) in special cases the circumcentric
dual lattice corresponds to the Voronoi lattice generated by the
$0$-skeleton of the simplicial lattice.

Once one has a dual lattice, discrete differential dual forms are
obtained analogously to the simplicial forms.  For a $p$-form $\omega$
and a $p$-element $\polyd{p}$ of the dual lattice, $\omega$ associated
to $\polyd{p}$ is given by
\begin{equation}\label{eq:DualForm}
\TSIP{\omega}{\polyd{p}}:= \omega(\polyd{p}) =\ \int_{\polyd{p}} \omega.
\end{equation}
This provides a way to directly discretize continuous forms on the
dual lattice.  Our next step is to show how to transform discrete
forms on one lattice to forms on the other.

The discrete Hodge (or $\star$) dual is an isomorphism
between the simplicial $k$-forms and the dual $(d-k)$-forms.  Given
that the definitions of the dual/simplicial forms are integrated
quantities over their respective lattice elements, we can assign an
average scalar density to the lattice element by dividing by the
volume of the lattice element.  The Hodge dual isomorphism then says
that these average scalar densities are equal for two dual elements;
\begin{equation}\label{eq:DisHodge}
  \frac{1}{A_{\polys{k}}} \TSIP{\omega}{\polys{k}}= \frac{1}{A_{{\star\polys{k}}} } 
\TSIP{\star \omega}{\star\polys{k}},
\end{equation}
where $A_{\polys{k}} = |\polys{k}|$ denotes the integrated measure,
i.e. norm, of $\polys{k}$. If we take the dual lattice to be the
circumcentric dual, then the relation in \refEQ{eq:DisHodge} says that
integration over a volume spanned by $\polys{k}$and $\star\polys{k}$
is preserved under the Hodge dual, i.e.  the discrete manifestation of
the self-adjointness of $\star$ in the $L^{2}$-inner product.

It is useful to note here a property of the Hodge dual in the lattice
that will be useful later on.  As in the continuum, if one scales a
differential $k$-form $\omega$ by a scalar $\alpha$, then the dual
$\star\left(\alpha\omega\right)$ has the same orientation as
$\star\omega$.  In the simplicial lattice the Hodge dual acting on a
subspace of a lattice element $\polys{k}$ returns back the lattice
element $\star\polys{k}$, possibly with a scalar coefficient.  In
particular, if one takes the dual of the portion of a triangle $\tris$
closest to an edge $\ls$, then one simply obtains the dual element to
$\tris$ itself.  The dual lattice element $\tris^*$ orthogonal to
$\tris$ lies on the face of the polytope $\ls^*$ and so entirely lies
in the convex hull of $\ls$ and $\ls^*$. Moreover it is orthogonal to
$\tris$ and therefore trivially orthogonal to the subspace of $\tris$
closest to $\ls$.

Given a general dual lattice to the simplicial lattice and the Hodge
dual, we can construct the complex
\begin{equation}\label{eq:SimpDualComplex}
\begin{CD}
0 @>\latcobound>>  \polys{0}   @>\latcobound>>  \cdots   @>\latcobound>>   \polys{d}  @>\latcobound>>  0 \\
@VV \star V      @VV \star V    @VV  \star V   @VV \star V  @VV \star V\\
0 @>\latbound>>  \polyd{d}   @>\latbound>>  \cdots   @>\latbound>>   \polyd{0}  @>\latbound>>  0
\end{CD}
\end{equation}
containing the chain and co-chain complexes in the dual and simplicial
lattices.  In the above complex, the operators $\latcobound$ and
$\latbound$ define the co-boundary and boundary operators,
respectively.

The next ingredient in the calculus of discrete differential forms is
the exterior derivative $d$ which maps $k$-forms to $(k+1)$-forms,
\begin{equation}\label{eq:DisExtDer}
  \TSIP{d \omega}{\polys{k+1}} = \TSIP{\omega}{\latbound \polys{k+1}} = 
\sum_{\polys{k} \in \polys{k+1}} \TSIP{\omega}{\polys{k}}. 
\end{equation}
Here we have used the lattice boundary operator (via Stokes' Theorem)
such that the action on a $k$-element of the simplicial $k$-skeleton
returns its $(k-1)$-boundary. The discrete exterior derivative allows
us to express $d\omega$ in terms of the valuations of $\omega$ on the
boundary of a given $\polys{k+1}$.  By replacing the
$\polys{k}$($\polys{k+1}$) with $\polyd{k}$ ($\polyd{k+1}$) we can
translate the discrete exterior derivative from simplicial forms to
dual forms.
$$
  \TSIP{d \omega}{\polyd{k+1}} = \TSIP{\omega}{\latbound \polyd{k+1}} = 
\sum_{\polyd{k} \in \polyd{k+1}} \TSIP{\omega}{\polyd{k}}.
$$

Similarly the exterior coderivative, $\delta = \star d \star$, on a
$(k-1)$-form can be derived by using the adjoint relationship (in the
continuum) between $d$ and $\delta$,
\begin{equation}\label{eq:DisExtCoDer}
  \TSIP{\delta \omega}{\polys{k-1}} = \TSIP{\omega}{\latcobound \polys{k-1}} = 
  \sum_{\polys{k} \ni \polys{k-1}} \TSIP{\omega}{\polys{k}}.
\end{equation}
Again, the same relationship holds on the dual forms
$$
  \TSIP{\delta \omega}{\polyd{k-1}} = \TSIP{\omega}{\latcobound \polyd{k-1}} = 
  \sum_{\polyd{k} \ni \polyd{k-1}} \TSIP{\omega}{\polyd{k}}.
$$
The co-derivative on discrete forms requires only our notion of the
Hodge dual and the exterior derivative and so the above expressions
come as a direct result of the adjoint relationships of these two
operators in $\TSIP{\cdot}{\cdot}$.  We can, in fact, rebuild much of
the algebra and calculus on exterior forms with these basic
building blocks.  The wedge product can also be reconstructed on the
discrete forms as shown by Desbrun {\em et al.} \cite{Hirani:DEC}.  We
do not reproduce the wedge product here, but only mention that this
wedge product will, in general, be non-associative in the discrete
scale.  Associativity, however, is recovered in the continuum limit.
One can construct an associative wedge product, but such a
construction will fail to preserve the anti-commutativity property.
It is generally a feature of discrete physics with finite angles that
certain continuum symmetries fail at the discrete level only to be
regained in the infinitesimal edge-length limit.

\section{The Local Structure of Discrete Forms}\label{sec:LocalForms}

We have described thus far the previous approaches to analysis of
curvature and differential forms on a \PFtri.  What was evident when
we examined the discrete forms and the relationship between forms on
the dual lattices was that the integrated measure of a lattice
differential form was preserved under the Hodge dual.  Given this
preservation, we can interpret the discrete forms as measures over a
volume spanned by a lattice element and its dual, i.e. the convex hull
of the vertexes of a lattice element and the vertexes of the dual
element.  This suggests a volume-based approach to discrete
forms. Since many of the operations applied to discrete forms rely on
transforming one discrete form on a $k$-skeleton to a discrete form on
a $p$-skeleton (where we may have $p\neq k$), we will be required to
transform objects from one domain to a non-coinciding (but possibly
overlapping) domain of integration.  We now turn to the local
properties of discrete forms and their inherent domains of support.
This will allow us to reconstruct transformations between lattice
elements in terms of hybrid simplicial-dual volumes, henceforth called
hybrid cells or hybrid volumes.

\subsection{A Menagerie of Hybrid Cells}\label{subsec:hybrids}

A hybrid volume is a domain that is at once the local measure of a
lattice element $\polys{k}$ and the orthogonal subspace dual to
$\polys{k}$.  For standard tensor analysis it is sometimes convenient
to take the $d$-simplexes as the local domains over which a function
or tensor is piecewise evaluated.  This is due to the ability to
define an unambiguous tangent space to any $d$-simplex in $\PFtri$.
However, an arbitrary $k$-form will, in general, be contained in
multiple $d$-simplexes and thus requires junction conditions to hold
across the multiple coordinate charts assigned to the $d$-simplexes
sharing a given $k$-simplex.  We will show how hybrid cells can be
viewed as natural domains for differential forms in $\PFtri$ such that
there are local orthogonal frames containing the carrier of the
geometric content of the discrete form. The lattice element hybrid
cells are atomized via local domains that contain the minimal,
non-trivial amount of information about the lattice $k$-skeleton.  We
build these hybrid cells from local constructions and discuss how they
are representative of the measure of the discrete forms.

We first construct domains that are shared by a set of $k$-simplexes,
one $k$-simplex for each $k = 0, \ldots, d$.  These shared domains
become the ``atoms'' of our geometry in the sense that these are the
simplest meaningful $d$-volumes in a PF manifold.  Of course one could
always subdivide these irreducible domains in some arbitrary way, even
to go so far as to define infinitesimal domains.  However, doing so
yields no further discrete information about the lattice or the
differential forms on the lattice.
\begin{defines}      %   [Irreducible Hybrid Cell]
\begin{subequations}
  An irreducible hybrid cell, $\irrhyb{\polys{0}\polys{1}\cdots
    \polys{d}}$, in a PF $d$-manifold $\PFtri$ is the $d$-simplex
$$
\text{\rm sgn}_{0\cdots d}\left([\circum{0}, \circum{1}, \ldots, \circum{d}]\right) = 
  \frac{1}{d!}\epsilon_{i_1 \cdots i_d} e^{i_1} \wedge \cdots \wedge e^{i_d},
$$
given that $\polys{k}\in \polys{k+1}$ for every $k<d$, $\epsilon_{i_1
  i_2 \cdots i_d}$ is the orientation of $\polys{d}$, and the vectors
$e_{i}$ are the vectors emanating from $\circum{0}$.  When a
circumcenter lies outside $\PFtri$, we only take the domain of
$\irrhyb{\polys{0}\polys{1}\cdots \polys{d}}$ that lies in $\PFtri$.
\end{subequations}
\end{defines}

Here a $k$-simplex, or convex hull of $k+1$ points, is denoted
$[a_{0}, a_{1}, \cdots, a_{k}]$ and $\circum{k}$ is the circumcenter
of the simplicial element $\polys{k}$.  The factor
$\text{sgn}_{0\cdots d}$ ensures the volume is consistent with the
induced orientation from $\polys{d}$, i.e. if the orientation of
$[\circum{0},\circum{1}, \ldots, \circum{d}]$ is opposite to that
induced in $\polys{d}$ then the orientation is flipped.  This is
matched by the totally-antisymmetric tensor $\epsilon_{i_1 i_2 \cdots
  i_d}$ whose orientation is induced by $\polys{d}$.  In cases where
the circumcenter of an element lies outside the element, the volume
gives negative contribution to sums over the irreducible hybrid
cell. The $2$-form still contains a vector with negative
$1$-dimensional orientation even after being made compatible with the
containing volume's orientation.  In general, the circumcenter lying
outside the simplicial element leads to an over-counting of volume.
This careful accounting of orientation ensures the conservation of
total volume, e.g. over the simplicial element $\polys{d}$.

It is clear from the definition that the $0$-skeleton of an irreducible
hybrid domain consists of the circumcenters of each of the
$k$-simplexes that share the domain.  In the $0$-skeleton only the
vertexes $\circum{0} = \polys{0}$ and $\circum{d} = \polyd{0}$ are
vertexes that are also members of  the simplicial or dual
skeletons.  Similarly, the $1$-skeleton of
$\irrhyb{\polys{0}\polys{1}\cdots \polys{d}}$ consists of $6$ vectors,
two of which are subspaces of either the simplicial or dual
$1$-skeleton of $\PFtri$. \refFig{Fig:FH} shows an irreducible hybrid
cell in three dimensions and highlights the members of the $1$-skeleton
that lie in either of the $1$-skeletons of $\PFtri$.
\begin{center}
\begin{figure}[t]
\includegraphics[width=2.5in]{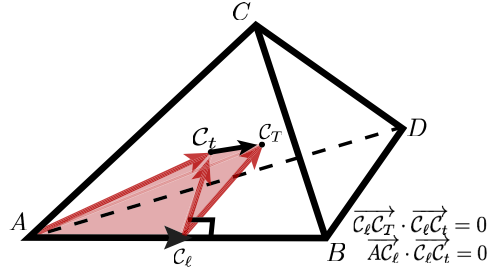}
\caption{An irreducible domain in three dimensions.  The domain (red,
  shaded) is that domain on which a $0$-form on $A$, a $1$-form on
  $[AB]= \ls$, a $2$-form on $[ABC]= \tris$, and a $3$-form on
  $[ABCD]=\tets$ are mutually defined.  The black edges bounding the
  irreducible domain are edges that lie in either the simplicial or
  dual $1$-skeleton of $\PFtri$.  The points ${\cal C}_{\ls}$, ${\cal
    C}_{\tris}$, and ${\cal C}_{\tets}$ label the circumcenters of
  $\ls$, $\tris$, and $\tets$ respectively.  Moreover, there is a
  natural orthogonal basis in this domain given by the edges
  $\protect\overrightarrow{\circum{\vs} \circum{\ls}}$,
  $\protect\overrightarrow{\circum{\ls} \circum{\tris}}$,
  $\protect\overrightarrow{\circum{\tris} \circum{\tets}}$.  This
  irreducible domain is foundation of the DEC as applied to the
  geometry of PF manifolds.  }  \label{Fig:FH}
\end{figure}
\end{center}

In each irreducible hybrid cell we can form an orthogonal (or
orthonormal) basis from the set of vectors $\{ \Solder{i}\ |
\Solder{i} = \overrightarrow{\circum{i} \circum{i+1}}\}$.  Moreover,
since each of these cells are subspaces of $\mathds{R}^{d}$, we have
the volume element given by $\tilde{V} = \frac{1}{d!} \epsilon_{i_{1}
  \cdots i_{d}} m^{i_{1}} \wedge \cdots \wedge m^{i_{d}}$.  In this
basis of differential forms, we carry the information about the
discrete $1$-forms $\ls \propto \Solder{0}$ and $\ld \propto
\Solder{d-1}$.  All other edges in the $0$-skeleton of this cell are
virtual in the sense that these $1$-forms only become basis elements
of higher-dimensional $k$-forms while not carrying direct information
about the discretization of $1$-forms on the lattice.  Similarly, the
irreducible hybrid cells contain discretization content for only one
simplicial and one dual $k$-form, for all $k$.  For simplicial
$k$-forms, this is clear from the definition.  For dual $k$-forms it
is evident from the restriction of the cell to exactly one
$(d-k)$-simplicial element.  The rest of the ${d+1 \choose k+1}$
$k$-forms in the irreducible hybrid cell carry the information about
the graded algebra in the simplex $\polys{d}$, but only indirectly
given the flat interior of the simplex.

The irreducible hybrid domains give the smallest domain of support for
any given $k$-form in the dual lattices of $\PFtri$.  From these
irreducible domains we wish to reconstitute local ``natural volumes''
associated to each lattice element of $\PFtri$.  The irreducible
hybrid cells do tile the PF manifold $\PFtri$ but will not generally
provide a disjoint cover of $\PFtri$.  Only when the triangulation is
well-centered (i.e. when the circumcenter of each simplicial element
is contained in the simplicial element) does the set of irreducible
hybrid cells form a disjoint cover.  We can now state a result with
regard to the measure of the hybrid cell.

\begin{thm}
  The convex hull (interior to $\PFtri$) of a simplicial element and
  its dual, $\text{CH}(\polys{k}, \star\polys{k})\cap \PFtri$, defines
  the domain of support for a discrete form on $\polys{k}$
  ($\star\polys{k}$) which is given by the set theoretic union of the
  irreducible hybrid cells (contained in $\PFtri$).  The volume
  measure of discrete forms on $\polys{k}$ ($\polys{k}$) is given by
  $V_{\polys{k}} = \frac{1}{{d\choose k}} |\polys{k}| \
  |\star\polys{k}|$.
\end{thm}
\begin{proof}
\begin{subequations}
  For each irreducible hybrid cell containing a given $\polys{k}$,
  $\irrhyb{\cdots \polys{k}\cdots}$, any $k$-form $\omega$ will
  generally have a non-zero evaluation over $\polys{k}$ and hence 
  have a non-zero component in each irreducible cell containing
  $\polys{k}$.  We can check that the set union of $\irrhyb{\cdots
    \polys{k} \cdots}$ is convex by examining the convex sum of
  extremal points on the $\irrhyb{\cdots \polys{k} \cdots}$.  First
  we examine the sum over $\polys{p}$ for $p<k$, which gives
$$
\bigcup_{\polys{0}, \ldots, \polys{k-1}} \left[ \circum{0}, \ldots \circum{k-1}, 
\circum{k}, \ldots, \circum{d}\right]\cap \PFtri  = [ \underbrace{\vs_{0}, \vs_{1}, \ldots, 
\vs_{k}}_{\polys{k}}, \circum{k+1}, \ldots, \circum{d}]\cap \PFtri,
$$
which is clearly convex.  Consider first the convex sum of
$\circum{p}$ and $\circum{p}'$ for two $\polys{p}$s ($p>k$) for which
the irreducible hybrid cell contain $\polys{k}$ and $\polys{p}$ is
non-zero.  We can examine the two irreducible hybrid cells who only
differ by $\polys{p}$.  There exists a boundary between the two cells
$[ \vs_{0}, \ldots,\vs_{k},\circum{k+1}, \ldots, \circum{p},\dots,
\circum{d}]$ and $[\vs_{0}, \ldots,\vs_{k},\circum{k+1}, \ldots,
\circum{p}',\dots, \circum{d}]$.  If $p=d$ then the boundary is a
subspace of the common $\polys{d-1}$.  In this case, the convex sum of
the two $\circum{d}$ and $\circum{d}'$ is the dual edge $\ld$, which is
a straight-line entirely contained on the combined domain (when the
two $d$-simplexes are mapped onto $\mathds{R}^{d}$). If $p<d$ then we
know from the circumcentric duality that the convex sum
$$\overrightarrow{\circum{p}\circum{p}'} =  \rho \circum{p} + 
(1-\rho) \circum{p}' \ \ \ \rho \in [0, 1]
$$
forms angles $\angle\overrightarrow{\circum{p-1}\circum{p}},
\overrightarrow{\circum{p}\circum{p}'} < \frac{\pi}{2}$ and
$\angle\overrightarrow{\circum{p}\circum{p}'},
\overrightarrow{\circum{p+1}\circum{p}} < \frac{\pi}{2}$ and hence
lies within both the irreducible hybrid cells.  We therefore conclude
that the set union of these irreducible cells for a given lattice
element is $\text{CH}(\polys{k}, \star\polys{k})$.

However, this entire domain may not contribute to the final oriented
sum over the $\irrhyb{\cdots \polys{k} \cdots}$s.  Two irreducible
hybrid cells covering the same domain with opposite orientations will
give zero contribution from the overlapping volume.  When summing over
the $\polys{p}$ for $p<k$, we simply get back a simplex
$$
\sum_{p<k} [\circum{0}, \ldots, \circum{k}, \ldots, \circum{d}] =
 [\underbrace{ \vs_{0}, \ldots, v_{k}}_{\polys{k}}, \circum{k+1}, \ldots, \circum{d}.
$$
Meanwhile  summing over the $\polys{p}$ for $p>k$ returns
$$
\sum_{p>k} [\circum{0}, \ldots, \circum{k}, \ldots, \circum{d}] = [\circum{0}, \ldots, \circum{k}, \{\polyd{d-k}\} ].
$$
Using these two results, the full sum gives  a bipyramid with base $\polyd{d-k}$ and $k$-dimensional altitude given by $\polys{k}$ whose volume is $V_{k} = \frac{1}{{d \choose k}} |\polys{k}|  \ |\star\polys{k}|$. 
\end{subequations}
\end{proof}

The hybrid domains associated to lattice elements of $\PFtri$ define
the local measure for the discrete forms on the lattice.  This domain
defined by the measure may not actually encompass the lattice element,
especially when the lattice is non-Pittway (when the element and its
dual have empty intersection).  However, the domain of support is
generally more expansive and necessarily contains the lattice element.
It is particularly insightful to notice that for a simplicial element
and its dual the domains of support and local measures for a discrete
form $\omega$ and its dual $\star \omega$ coincide.  These
lattice-element hybrid measures and the convex domain of support then
become fundamental to the discretization and algebra of discrete
forms.  Henceforth, when we think of the hybrid cell, we will take
this to mean the local lattice measure and not the domain of compact
support, since it is only the former that is necessary for explicit
computations. Examples of the menagerie of the local measure hybrid
cells in $3$ dimensions is shown in \refFig{fig:hybcell}.
\begin{center}
\begin{figure}[h]
\includegraphics[width=3in]{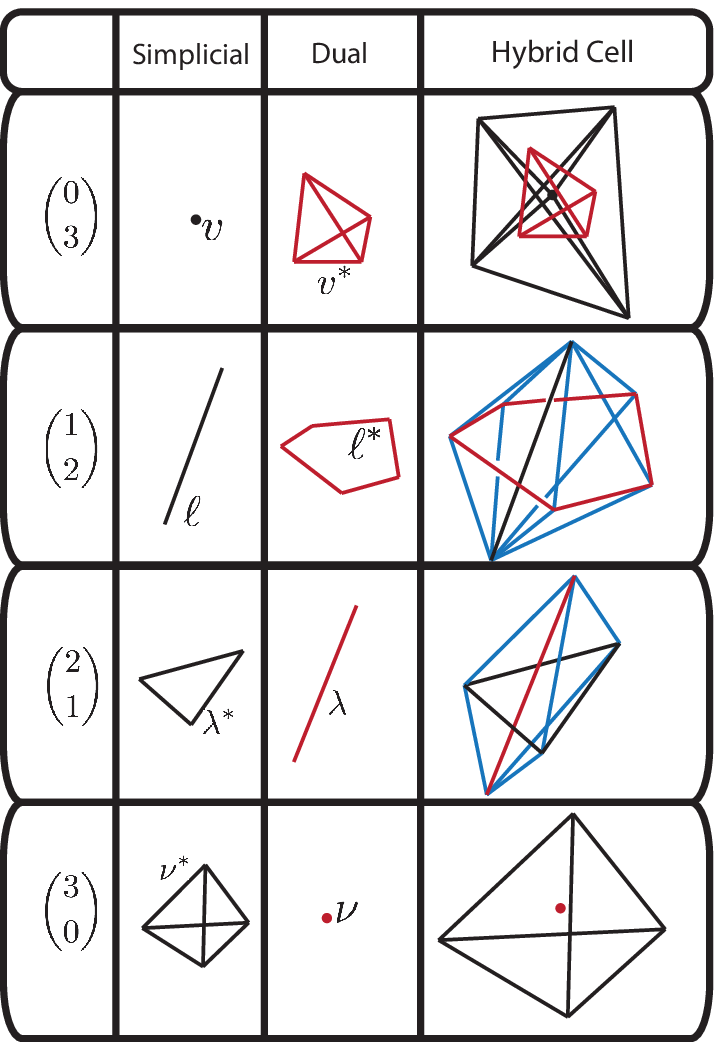}
\caption{In three dimensions there are four distinct classes of hybrid
  cells in $\PFtri$.  We show in this figure representative depictions
  of these four classes.  The labeling ${ k \choose (d-k)}$ labels
  the dimension of a $k$ simplicial element$\polys{k}$ and its dual.
  The hybrid cell is heuristically constructed by connecting the
  vertexes of $\polys{k}$ with the vertexes of the dual cell
  $\star\polys{k}$.  In the cases of ${0 \choose d}$ and ${d \choose
    0}$, the hybrid cell is just equal to $\polys{d}$ or $\polyd{d}$,
  respectively.  }\label{fig:hybcell}
\end{figure}
\end{center}

We have built up the lattice element hybrid cells from irreducible
domains. We now take a step in the reverse direction to examine hybrid
cells common to multiple lattice elements, not just a given lattice
element.  Such constructions are very useful when we are required to
relate discrete forms on a $k$-skeleton to discrete forms from a
$p$-skeleton ($p\neq k$).  For example, the exterior derivative
$d\omega$ is a map from $k$-forms to $(k+1)$-forms and requires a
relationship to be formed between hybrid cells that overlap but do not
coincide.  The integration over the hybrid domain to $\polys{k+1}$
picks up contributions from each of the $\polys{k}\in \polys{k+1}$,
but they do not contribute equally.  Rather the ``democratic''
allotment of domains by the hybrid cells for each $\polys{k}$ defines
a domain of support for each of the $\polys{k}$.  The integration over
the hybrid domain for the $\polys{k+1}$ then acts as a restriction on
the integration for the $\polys{k}$ and we find a domain common to
$\polys{k}$ and $\polys{k+1}$ for each term in the discrete $d\omega$.

\begin{cor}
\begin{subequations}
  The hybrid cell common to two simplicial elements $s$ and $s'$
  (with $\text{dim}(s)<\text{dim}(s')$) is given by sum of oriented volumes
\begin{align}\label{eq:CorHyb}
V_{s s'} =  &\sum_{\polys{k} \neq s,s'}   \irrhyb{\polys{0} \cdots s \cdots {s'} \cdots 
\polys{d}} \nonumber  \\
&\ +\sum_{\polys{k}|_{s}}\sum_{\polys{m}|_{s'}} \irrhyb{\polys{0} \cdots s \cdots \polys{d}} 
\cap  \irrhyb{\polys{0} \cdots s' \cdots \polys{d}} .
\end{align}
  When the simplicial lattice is well-centered (i.e. $c(\polys{d}) \in
  \polys{d}, \ \forall \polys{d}$), then the hybrid cell common to a
  set of simplicial elements ${\cal S} =
  \polys{i_1}_{1},\polys{i_2}_{2}, \ldots, \polys{i_{n}}_{n}$ is given
  by the sum;
\begin{align} 
V_{1, \ldots, n}  = \sum_{\polys{k} \notin{\cal S}}   V_{\polys{0} \cdots \polys{i_{1}} \cdots \polys{i_n} \cdots \polys{d}}.
\end{align}
\end{subequations}
\end{cor}
\begin{proof}
  The hybrid domain common to any two simplicial elements is obtained
  by the intersection of the hybrid cells for each element $s$ and
  $s'$.  This reduces the problem to pairwise set intersections over
  irreducible hybrid domains.  For an irreducible hybrid cell
  containing both $s$ and $s'$, the intersection returns the full
  irreducible cell.  Any irreducible hybrid cell for $s$ that is not
  also a hybrid cell for $s'$ will yield the subspace common to both
  irreducible cells, but with an orientation opposite of the
  simplicial complex.  These contribute with negative volume and
  remove subspaces not in any shared $\polys{d} \ni s, s'$.  Hence,
  the hybrid cell common to $s$ and $s'$ is reduced to a sum over 
  irreducible hybrid cells plus negative volume terms obtained by
  intersections of non-shared irreducible cells.  The extension to $n$
  simplicial elements is straightforward, though quickly becomes
  cumbersome.  In the case of well-centered simplicial complexes, all
  irreducible hybrid cells are disjoint and trivially factorize any
  given $\polys{d}$.  Hence, the second summation \refEQ{eq:CorHyb}
  vanishes.
\end{proof}

The above definitions have only focused on hybrid cells for simplicial
lattice elements.  As we noted before for the hybrid cell for a given
lattice element, the hybrid domain for $\polys{k}$ is coincident with
the hybrid domain for $\star\polys{k}$ as a result of the
self-adjointness of the Hodge dual in the $L^{2}$-inner
product. Generalization from simplicial elements for the reduced
hybrid cells or domains common to multiple lattice elements is,
therefore, trivial by simply taking the dual of an element in the dual
lattice.

\subsection{Solder Forms and Moment Arms}\label{subsec:solder}

We consider a frame bundle on our manifold.  In each simplex we have a
fibre that is a copy of the tangent space on the base manifold, which
we call the space of values. The tangent
space of the base manifold is the horizontal section while the fibre
or space of values is the vertical section.  On $\PFtri$ we assign a
tangent space to any $d$-simplex or any pair of neighboring
$d$-simplexes.  For each of these tangent spaces, we have a copy in
the space of values.

The PF manifold has a hard-wired simplicial skeleton and from that we
construct a dual skeleton, entirely determined by its rigid predecessor
lattice.  We can think of these two lattices as complements of one
another much the way we consider the space of vectors and one-forms as
complements.  Moreover, since the dual lattice is decomposed into
subspaces of $d$-simplexes, we can define the dual lattice entirely in
terms of the vectors in the tangent spaces. Thus, determining
transformations between the simplicial and dual lattices is tantamount
to determining an appropriate solder form in the $d$-simplexes, or
more appropriately in the irreducible hybrid cells.  A solder form,
the unit vector-valued one-form, is the identity map from the tangent
space--with one-form basis $\TSForm{a}$--to the space of values--with
vector basis $\sovVec{a}$;
\begin{equation}
d{\cal P} = \sovVec{a}\TSForm{a}.
\end{equation} 
When acting on a vector $d{\cal P}$ transforms a vector in the tangent
space to a vector in the space of values.  The summation ensures that
we retain the basis expansion but in the new vector space.  We can now
construct a representation of the solder form in $\PFtri$ that allows
for transformations between the simplicial and dual lattices.

In the discrete manifold, the solder form is an object that carries
information about the relationship between the space of values
(vertical section) and the tangent space (horizontal section).  In
conjunction with the Hodge dual, the solder form acts as a transform
between $k$-forms, $\omega$, and $p$-forms ($p\leq d-k$) orthogonal to
$\omega$.  Through wedge products we can form a $(d-p)$-dimensional
space orthogonal to a desired $p$-form that contains our $k$-form.
The Hodge dual and a summation over discrete $k$-forms that are able
to construct such a space then provide a transformation between a set
of $k$-forms and the desired $p$-form.  Another route is to take the
Hodge dual of $k$-forms orthogonal to a desired $p$-form and contract
over the directions orthogonal to the $p$-form.  Together these paths
form two routes towards a notion of the double-dual of a vector-space
valued differential (discrete) form.

We start the construction in an irreducible hybrid cell.  In this
irreducible cell we have exactly one $\ls$ and $\ld$ that are
representative edges of the simplicial and dual 1-skeletons,
respectively. As discussed above, this domain has volume form given by
\begin{equation}
\irrhyb{\polys{0}, \ls, \cdots, \ld, \polys{d}} = \frac{1}{d(d-1)} 
\left(\ls \wedge \MomArm{\ls}{\ld}\wedge \ld\right),
\end{equation}
where the $\MomArm{\ls}{\ld} = \frac{1}{(d-2)!} \Solder{1} \wedge
\cdots \wedge \Solder{d-2}$ is the $(d-2)$-dimensional subspace
orthogonal to both $\ls$ and $\ld$, which we call the moment arm from
$\ls$ to $\ld$.  Both $\ls$ and $\ld$ are understood to consist of
only the segment of $\ls$ and $\ld$ within a given irreducible hybrid
cell.  Taking the set $\{\ls, \ld, \{\Solder{i}\}_{i=1}^{d-2}\}$ as an
orthonormal basis (where $\{\Solder{i}\}$ are the
one-forms that span $\MomArm{\ls}{\ld}$), we have the
relationship between the lattice forms and their induced vectors;
\begin{subequations}
\begin{align}
&\ls(\ls) = 1, \ \ \ \ \ \ \ \ \ \  \ \ \ \    \ls(\ld) = 0, \ \ \ \ \ \ \  \ \ 
\ls(\Solder{i}) = 0\\
&\ld(\ld) = 1, \ \ \ \ \ \ \ \ \ \  \ \ \  \ld(\ls) = 0, \ \ \ \ \ \ \ \    
\ld(\Solder{i}) = 0\\
&\Solder{i}(\Solder{i}) = 1, \ \ \ \ \ \ \ \ \ \  
\Solder{i}(\ls) = 0, \ \ \ \ \ \  \Solder{i}(\ld) = 0,
\end{align}
\end{subequations}
where $\MomArm{\ls}{\ld}$ represents any $1$-form or vector in the
subspace defined by $\MomArm{\ls}{\ld}$.  We can then identify
individual maps from each basis form to an identified basis element in
the vertical section (space of values).  In general, we make no real
distinction between the lattice element as a scalar-valued
differential form or as a vector-space valued differential form.  In the
context of this manuscript, the lattice elements always take the
meaning of a vector-valued differential form, or as a map from the
tangent space to the space of values.  The solder form is thus
\begin{equation}
d{\cal P}_{0} = \ls+\ld(\polys{d})+\sum_{m_{1}}^{m_{(d-2)}}
\Solder{i}(\polys{d}).
\end{equation}

The combination of the solder form and the Hodge dual provide us with
practical tools for transforming between the two lattices.  We first
examine the use of the solder form in such a transformation that is
well-known in general relativity.  In E. Cartan's \cite{Cartan}
approach to general relativity, the Einstein tensor need not be
defined with relation to the Ricci curvature and Ricci scalar, but as
the dual to the moment of rotation;
\begin{align}\label{eq:MoR}
{\bf G} & \equiv \star \left(  d{\cal P} \wedge {\cal R}\right) \nonumber \\
 & = \star\left(\frac{1}{4}\sovVec{\mu}\delta^{\mu}_{\ \nu} \TSForm{\nu} \wedge 
       \sovVec{\sigma}\wedge \sovVec{\tau} R^{\sigma\tau}_{\ \ \alpha\beta} \TSForm{\alpha} 
      \wedge \TSForm{\beta}\right)  \nonumber \\
& = \frac{1}{4} \sovVec{\xi} \tilde{\epsilon}^{\xi}_{\ \mu\sigma\tau}R^{\sigma\tau}_{\ \ \ \alpha\beta} 
    \TSForm{\mu}\wedge \TSForm{\alpha}\wedge \TSForm{\beta}\nonumber \\
& = \sovVec{\xi} \underbrace{ \frac{1}{4}\tilde{\epsilon}^{\xi}_{\ \mu\sigma\tau}R^{\sigma\tau}_{\ \ \ \alpha\beta}
    \epsilon_{\zeta}^{\ \mu\alpha\beta}}_{G^{\xi}_{\ \zeta} = (*{\cal R}*)^{\xi \rho}_{\ \ \zeta\rho}}d\Sigma^{\zeta},
\end{align}
where we have specifically worked in dimension $4$.  Here the solder
form plays the crucial role of defining a subspace orthogonal to both
the 3-volume for the moment of rotation and imposing the trace on the
$3$-form that results from the wedge product.  In higher dimensions,
we keep appending onto the Riemann curvature $(d-3)$ solder forms
prior to taking the Hodge dual to obtain the moment of rotation
$(d-1)$-form.  In this light, the solder forms take on the
interpretation as moment arms between the original basis elements and
elements in the orthogonal subspace.  This map provides a way to
identify mappings between the dual lattices not provided by the Hodge
dual alone.

The solder form is an instrumental tool in the transformation of
vector-space valued differential $k$-forms to differential $p$-forms.
Of particular interest to us now is the transformation of $1$-forms
$\omega\in \Lambda^{(1)}$ to $1$-forms in the dual space $\omega' \in
\Lambda^{*(1)}$, which is the raising (or lowering) operation on
differential forms. We take a lattice $1$-form $\ls$ and construct a
map to a dual one-form $\ld$.  This is equivalent to mapping the
$1$-form $\ls$ to the dual space of $1$-forms and asking for the
components along $\ld$.

To make such a map, we first make use of a discrete analog of a
continuum property relating the Hodge dual, the wedge product $\wedge$
and the inner derivative $\iota$.
\begin{thm}\label{thm:CD1}
  Let $\PFtri$ be a PF simplicial manifold and choose a simplicial
  $k$-element $\polys{k}\in \PFtri$, a one-form $\Solder{k}$ that
  extends from $\circum{k}$ to $\circum{k+1}$ of a $\polys{k+1}\ni
  \polys{k}$, and a circumcentric Hodge dual operator, $\star$, on
  $\PFtri$.  Then there exists a commutative diagram
\begin{equation}\label{eq:DD1}
\begin{CD}
\Lambda^{(k)} @>\wedge_{\Solder{i}}>> \Lambda^{(k+1)}\\
@VV\star V                                        @VV\star V\\
\Lambda^{*(d-k)}  @>  (d-1)\iota_{\Solder{i}}>> \Lambda^{*(d-k-1)},
\end{CD}
\end{equation}
using the wedge product with $\Solder{k}$ and the inner derivative
(contraction), $\iota_{\Solder{k}}$, over the vector,
$\overrightarrow{\Solder{k}}$, corresponding to $\Solder{k}$.
\end{thm}
\begin{proof}
\begin{subequations}
  We first prove this for the case $k=1$.   Using the definition of the circumcentric Hodge dual and its property that any subspace of a $p$-simplex, $\polys{p}$, gives back $\star\polys{p}$ yields 
\begin{align}
\star\left( \Solder{1} \wedge \ls\right)   &= 
   \sum_{\polys{p} \ni \ls, \tris:  \ p>2} \text{sgn} \left[ \circum{2}, \circum{3}, \ldots, \circum{d} \right] = \sum_{\Solder{i}: \  i>3} \frac{1}{(d-2)!} \Solder{2} \wedge \cdots \wedge \Solder{d-1},
\end{align}
where the summations are over the $\polys{k}$ (for $k>3$) containing $\ls$. 
Meanwhile, taking the Hodge dual first then the inner derivative results in
\begin{align}
\iota_{\Solder{1}} \left( \star \ls\right) & = \iota_{\Solder{1}} \left( \sum_{\Solder{j}: \ j>1}\frac{1}{(d-1)!}  \Solder{1} \wedge \cdots \wedge \Solder{d}\right) \nonumber \\
& =  \sum_{\Solder{j}: \ j>2}\frac{1}{(d-1)!}  \Solder{2} \wedge \cdots \wedge \Solder{d}
\end{align}
where we use the normalization property $\Solder{i}\left(
  \overrightarrow{\Solder{j}}\right) = \delta_{ij}$ that picks out a
given $\polys{2}$ from the summation.  Multiplying by $(d-1)$ yields
the desired result.

The case for arbitrary $k$ follows in a straightforward manner.
\end{subequations}
\end{proof}

This theorem shows that there exist two paths to map from an arbitrary
$k$-form on the simplicial lattice to a $(d-k-1)$-form on the dual
lattice. These two paths stem from continuum descriptions of the
double-dual transformations of tensors (as was used to construct the
Einstein tensor).  Since discrete forms are mapped as coefficients on
the lattice elements, we need only know how the space of forms on the
lattice elements get mapped to one another.  The scalar coefficient
gets carried through with no change.

As a result of Theorem~\ref{thm:CD1}, we can show a subsequent commutative diagram for maps from $\Lambda^{(1)}$ to $\Lambda^{* (1)}$.   We will drop combinatoric factors from the commutative diagrams for simplicity.      
\begin{cor}\label{cor:CD2}
On $\PFtri$ there exist maps from the simplicial $1$-skeleton to the dual $1$-skeleton such that
\begin{equation}\label{eq:DD2}
\begin{CD}
\Lambda^{(1)} @>\wedge^{(d-2)}>> \Lambda^{(d-1)}\\
@VV\star V                                        @VV\star V\\
\Lambda^{*(d-1)}  @>\iota^{(d-2)}>> \Lambda^{*(1)}.
\end{CD}
\end{equation}
\end{cor}
\begin{proof}
  The proof follows by induction and successive applications of
  Theorem~\ref{thm:CD1} as in the diagram below:
\begin{equation}\label{eq:DD3}
\begin{CD}
\Lambda^{(1)}         @>\wedge_{m_{1}}>>      \Lambda^{(2)}                @>\wedge_{m_{2}}>>      \Lambda^{(3)}    \\
@V\star VV                                                  @VV\star V                               @VV\star V\\
\Lambda^{*(d-1)}     @>\iota_{m_{1}}>>          \Lambda^{*(d-2)}         @>\iota_{m_{2}}>>          \Lambda^{*(d-3)}.
\end{CD}
\end{equation}
Iteratively applying the above diagram, we finally construct the desired result. 
\end{proof}

The two paths in \refEQ{eq:DD2} define our two notions of the
double-dual of a vector-space valued differential form. The Down-right
path acts as the trace of the double dual to obtain a vector oriented
in a given direction, while the Right-down path builds a $(d-1)$
subspace orthogonal to the desired vector before taking the dual to
obtain the desired result.  The Right-down path corresponds to the
first line of \refEQ{eq:MoR} while the Down-right path corresponds to
the last line of \refEQ{eq:MoR}.

Taking the inner-derivative of the hybrid volume form $V_{\ls\ld}$
with a desired $\ld$ gives
\begin{equation}
 d\iota_{\ld} V_{\ls\ld} = \frac{1}{(d-1)} \ls\wedge \MomArm{\ls}{\ld},
\end{equation}
which defines the orthogonal subspace to $\ld$.  We find the moment
arm by taking the inner-derivative again,
\begin{equation}
d(d-1)\iota_{\ls}\iota_{\ld} V_{\ls\ld} = \MomArm{\ls}{\ld}.
\end{equation}
The inner-derivatives are used to find the components of the moment
arm that maps $\ls$ to $\ld$ (or the reverse).  Choosing another $\ls$
or another $\ld$ changes the moment arm.  Using our above commutative
diagrams, we then have
\begin{equation}\label{eq:DD5}
\begin{CD}
\Lambda^{(1)} @>\wedge^{\MomArm{\ls}{\ld}}>> \Lambda^{(d-1)}\\
@VV\star V                                        @VV\star V\\
\Lambda^{*(d-1)}  @>\iota^{\MomArm{\ls}{\ld}}>> \Lambda^{*(1)}.
\end{CD}
\end{equation}
Given that these moment arms play the same role as the solder forms in
mapping $k$-forms to $p$-forms, we will often refer to the moment arms
as generalized solder forms.  Just as the Einstein tensor can be
expressed as
$$ {\bf G} = \star \left( \underbrace{d{\cal P}\wedge \cdots \wedge d{\cal P}}_{d-3\text{-times}}
  \wedge {\cal R}\right) =\text{Tr}\left(*{\cal R}*\right),
$$ 
using the
solder forms to map the $2$-form to a $1$-form, the moment arms
$\MomArm{\ls}{\ld}$ allow us to define a map from the simplicial
$1$-skeleton to the dual $1$-skeleton on the dual lattice.  In the
more general case (\refFig{fig:solder}) of mapping a simplicial
$k$-form $s$ to a dual $p$-form $\sigma$ ($p<d-k$), we define the
moment arm $\MomArm{s}{\sigma}$
\begin{equation} \label{eq:SimpDualSolder} 
\MomArm{s \sigma} = {d \choose  k}
   {d-k \choose p} {\iota_{s}\iota_{\sigma}}V_{s\sigma} =  
  \frac{d!}{k!(d-k-p)! p!}   {\iota_{s}\iota_{\sigma}}V_{s\sigma}.
\end{equation}
This is the ${k \choose p}$-solder form between the simplicial and
dual lattices.  Of course, this solder form only makes sense when
$\sigma$ and $s$ have overlapping domains, i.e. when $s \in \star
\sigma$ or $\sigma \in \star s$.  Moreover, it is a natural
consequence that the solder form between a lattice element and its
dual is given by the point of intersection, and so we have a
consistent framework to map between the dual and simplicial lattice
elements.

We can also construct solder forms between simplicial lattice
elements.  Given the duality between $\sigma$ and a simplicial element
$\polys{d-p}$, \refEQ{eq:SimpDualSolder} also defines a moment arm
between two simplicial elements.  If we have two simplicial elements,
$s\in\{ \polys{k} \}$ and $s'\in \{\polys{p}\}$ then the moment arm or
solder form between $s$ and $s'$ is given by
\begin{equation} \label{eq:SimpSimpSolder}
%%%  k is the dimension of s''
\MomArm{s}{s'} = {d \choose k} {d-k \choose d-p}  {\iota_{s} \iota_{\star s'}} V_{\star s, s'}.
\end{equation}
\begin{center}
\begin{figure}[h]
\includegraphics[width=3.33in]{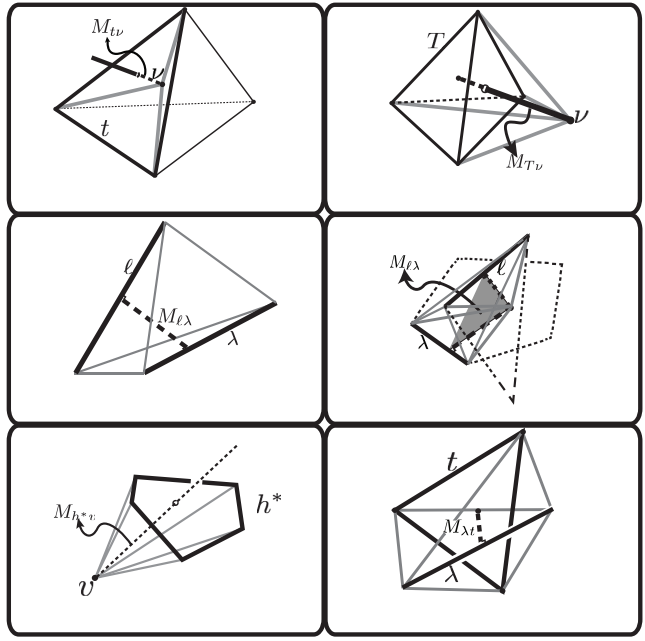}
\caption{Solder Forms are shown in three ({\em left}) and four ({\em right})
  dimensions.  When the dimensionality of the space spanned by two
  elements $\polys{k}$ and $\polyd{p}$ is $(d-1)$, a solder form
  between a $\polys{k}$ and a $\polyd{p}$ is given by the vector from
  $\circum{\polys{k}}$ to $\circum{\star\polyd{p}}$.  In more general
  cases, the generalized solder form between two elements is given by
  the $(d-p-k)$ subspace orthogonal to both $\polys{k}$ and
  $\polyd{p}$.  }\label{fig:solder}
\end{figure}
\end{center}

A similar result holds for solder forms between two dual lattice
elements.  This provides a geometric foundation for measures of
discrete forms.  The irreducible, reduced and standard hybrid cells
form a topological framework on which we form local measures for
lattice elements, while the moment arms/solder forms allow us freedom
to transform from one lattice to another.  This has yet to tell us
anything about the structure of discrete differential forms in this
context, and so we now shift our focus to the exterior calculus using
the hybrid domains.

\subsection{The Algebraic Structure of Forms on Hybrid Cells}\label{subsec:alghyrbid}
%%% The mixed wedge product (dual-simplicial) is a natural formation
%%% from the soldering forms, but the primitive wedge product
%%% (simplicial-simplicial or dual-dual) has unique properties that
%%% come from the dispersion of the distributions across the hybrid
%%% cells (inclusive of the boundary)

A  common property of exterior calculus on discrete
manifolds is the distributional nature of the discrete forms.  In
the canonical approach to DEC, we have discussed how the differential
forms are obtained by integration of the continuous differential form
over the simplicial element in $\TriMan$ corresponding to a simplicial
element in $\PFtri$.  These discrete differential forms  take on
values when evaluated on the simplicial or dual skeletons.  

In applying a DEC formalism to curvature operators on the PF
manifolds, we want to utilize the known properties of curvature in RC
and use as our guide the Regge action principle.  It is known from the
canonical, continuum analysis of PF curvature that any loop of
parallel transport in a plane orthogonal to a hinge will non-trivially
transform tangent vectors carried around the loop.  This occurs even
when the loop of parallel transport does not intersect the orthogonal
dual polygon to a hinge (independent of how one defines the dual
lattice).  As a result, the curvature remains non-zero as long as the
loop of parallel transport contains a non-trivial projection onto a
surface parallel to the dual polygon.  Therefore, the discrete measure
of the curvature in the hybrid domain of a hinge is given by
\begin{align}  \label{eq:EHR}
R \rightarrow \int_{V_{h}} R\ dV_{\rm proper}= \frac{1}{{d\choose 2}}\int_{h} \int_{h^{*}}
\left({\bf Riem}\cdot dh^{*}\right) dh = 2\defect{h} A_h,
\end{align}
where the combinatoric normalization comes from the decomposition of
the volume.  We then notice that this is the total curvature across
the hybrid cell.  Yet it has little direct indication of the local
tensorial content of the Riemann tensor.  From a DEC perspective, we
wish to find a projection of ${\bf Riem}$ into $\PFtri$ that preserves
this total curvature. 
  
Since the curvature in PF manifolds is entirely projected onto the
polygonal dual $\hinge^*$ to $\hinge$ we then require that this result
be constant over $\hinge$ in order to recover \refEQ{eq:EHR}. This is
a direct indication that the discrete form defines a constant field
over $\hinge$.  If we take the Hodge dual on the discrete forms,
then we also have,
$$ \int_{h} \left[ \int_{h^{*}} {\bf Riem}\cdot dh^{*}\right] dh = 2\defect{h} A_h   
     = \int_{h^{*}} \left[ \int_{h} \star{\bf Riem} \cdot dh \right]. 
$$
The scalar coefficient is thus constant over the hybrid domain, and we
now define an approach based on the standard DEC that explicitly
assigns the discrete forms over the hybrid measures.

In an explicit volume-based DEC, the base discretization is done as
before by projection of the differential $\omega$ onto a simplex or
dual polytope in $\TriMan$ that corresponds to a simplex or dual polytope
in $\PFtri$,
$$ \omega_{\polys{k}}  = \frac{\TSIP{\omega}{\polys{k}}}{ |\polys{k}|}.$$
Our discrete forms then become the geometric objects
$\omega_{\polys{k}} \polys{k}$ where $\polys{k}$'s are the $k$-forms
in the hybrid cells for $\polys{k}$.  The $\polys{k}$ does not merely
represent the lattice element, but the family of surfaces in
$V_{\polys{k}}$ parallel to $\polys{k}$, each such surface mapped back
to $\polys{k}$ when viewed from outside $V_{\polys{k}}$, i.e. when
coarse-grained to smooth over the internal structure.  As a comparison
to the continuum, we have the relationship
$$
\LIP{\omega}{\polys{k}} = \int_{V} \omega_{\polys{k}} dV,
$$
using our previous definition of $\omega_{\polys{k}}$.  Then, using the
projection of $\omega$ onto a lattice element,
$\TSIP{\omega}{\polys{k}} = \int\omega_{\polys{k}}{\polys{k}}$ we have
\begin{equation}
\LIP{\omega}{\polys{k}} = \frac{1}{{d\choose k}} \int_{\star\polys{k}}
\TSIP{\omega}{\polys{k}} d(\star\polys{k}).
\end{equation}
We then take the coefficient $\omega_{\polys{k}}$ as a scalar function
defined over the dual polytope $\star\polys{k}$ but with components
only lying in the surfaces parallel to $\polys{k}$.  The discrete
measure of a $k$-form $\omega$ is
\begin{equation}\label{eq:L2Discrete}
  \LIP{\omega}{\polys{k}} : = \frac{1}{{d\choose k}}\int_{V_{\polys{k}}} \omega_{\polys{k}} \polys{k} \wedge \star\polys{k}.
\end{equation}
Further, this volume measure has the standard property that the Hodge dual preserves the coefficient
and so we have
\begin{equation}\label{eq:L2DiscreteHodge}
\LIP{\star\omega}{ \star\polys{k}} = \LIP{\omega}{\star(\star\polys{k})} = \LIP{\omega}{\polys{k}}.
\end{equation}
 The definitions of the exterior derivative and
co-derivative follow  from Stokes' theorem applied to the $L^{2}$-inner product;
\begin{align}\label{eq:L2ExDer}
\LIP{d\omega}{\polys{k+1}} &= \frac{1}{{d\choose k+1}}\int_{V_{\polys{k+1}}}  
\TSIP{d\omega}{\polys{k+1}} d\polys{k+1} d(\star\polys{k+1})  \nonumber \\
&= \frac{1}{{d\choose k+1}}\frac{1}{k+1} \int_{V_{\polys{k+1}}}\sum_{\polys{k}} \TSIP{\omega}{\polys{k}} d\polys{k} d\MomArm{\polys{k}}{\polys{k+1}} d\star\polys{k+1}
 = \LIP{\omega}{\latbound \polys{k+1}}|_{\polys{k}}.
\end{align}
where we have decomposed $\polys{k+1}$ into $\polys{k}$'s and the
moment arms $\MomArm{\polys{k}}{\polys{k+1}}$.  Moreover, the integral
  is maintained over the measure corresponding to $\polys{k+1}$ and so
  the total volume integral is over the volume
  $V_{\polys{k}\polys{k+1}}$.  After integrating and combining
  appropriate terms, we then have
\begin{equation}
\LIP{d\omega}{\polys{k+1}}  =  \frac{1}{{d\choose k+1}}\frac{1}{ (k+1)} \sum_{\polys{k} \in \polys{k+1}} 
\omega_{\polys{k}} |\polys{k}| |m_{\polys{k}{\star\polys{k+1}}}| |\star\polys{k+1}|  
=\sum_{\polys{k}} \omega_{\polys{k}} V_{\polys{k}\polys{k+1}}
\end{equation}
Using similar properties of the co-derivative in the local
inner-product, we obtain the measure of the discrete exterior
co-derivative;
\begin{equation}
\LIP{\delta \omega}{\polys{k-1}}  =  \sum_{\polys{k} \ni  \polys{k-1}} \omega_{\polys{k}} V_{\polys{k}\polys{k-1}}.
\end{equation}

Some key points and assumptions are useful to highlight before
directly using this formalism.  A core assumption for the DEC
approach to vector-space valued differential forms is that we only
explicitly discretize the base manifold.  In this sense, we only
discretize the tangent space components and examine the constraints
this puts on components in the space of values.  This is crucial for
understanding the role of vector-space valued forms in the lattice.
For a basis of the discrete forms, we expand a $k$-chain in terms of
the lattice elements, while the basis in the space of values is given
by the unit $p$-forms with magnitude given by the coefficients defined
in the discretization.  Secondly, while discretization is done on a
given lattice element, the $k$-form field is extended as a constant
field domain of support for the $k$-element and measured on the hybrid
cell.  Under a standard assumption of the assignment of a field value
to the lattice element, the discrete form would be distributionally
valued on the lattice element. However, our measure requires the
valuation to be constant across the orthogonal subspace to the lattice
element.  This is essential to properly recover local integral
measures as averages over a given finite domain.

We have provided a scheme based on standard DEC where the measures of
$k$-forms are given by integral $d$-measures instead of local
$k$-measures on $k$-elements.  This has the advantage of simplifying
our evaluation of the discrete forms and providing a consistent
framework for the understanding of the Hilbert action. However, we
have only specified how one takes differential forms and discretizes
them, but not how the inherent properties of the lattice affect the
coefficients in the discretization.  To examine the DEC approach to
curvature in PF manifolds, we must examine properties of local
curvature operators in addition to the integral measures provided by
the Hilbert action.  In the next section we analyze the standard
curvature operators and the manifold geometry from the joint
perspective of DEC and RC.

\section{Curvature Forms in Piecewise-flat
  Manifolds}\label{sec:CurvPF}

Having laid out the foundations for analyzing curvature we can now
provide measures of local components for the standard curvature
operators.  In particular, we examine the nature of curvature in the
hybrid domain to codimension 2 hinges, simplicial edges, and
simplicial vertexes.  We have already discussed the role of the
Hilbert action in measuring the curvature.  We now specify how this
can be obtained and used to derive local curvature operators.

The Riemann tensor is defined as a ${1 \choose 1}$-tensor valued
$2$-form, or (by the  raising operation in the space of values) as a
bivector-valued $2$-form;
\begin{align}
{\bf Riem} = \frac{1}{2} \sovVec{\mu} \wedge \sovForm{\nu} R^{\mu}_{\
  \ \nu \sigma\tau} \TSForm{\sigma}\wedge \TSForm{\tau} = \frac{1}{4}
\sovVec{\mu}\wedge \sovVec{\nu} \underbrace{R^{\mu\nu}_{\ \ \
    \sigma\tau}}_{g^{\nu\alpha}R^{\mu}_{\ \ \alpha \sigma\tau}}
\TSForm{\sigma}\wedge \TSForm{\tau}.
\end{align}
The Riemann tensor thus takes as an argument a bivector from the base
manifold to return a rotation operator or rotation bivector in the
space of values, i.e. for a loop of parallel transport enclosing a
given area an arbitrary vector ${\bf A }= \sovVec{\mu}A^{\mu}$ will be
transformed as ${\bf A} \rightarrow {\bf A}' = {\bf A} + {\bf \delta
  A}$ with
$$
\delta A^{\mu} = -R^{\mu}_{\ \ \nu \sigma \tau}A^{\nu}
\Sigma^{\sigma\tau}.
$$ 
In the discretization, the magnitude of $\delta{\bf A}$ will be
determined by the geometry, while the basis bivectors in the space of
values will remain a set of orthonormal bivectors.  We can then set
the basis of the tangent space as the integrated measures of the
lattice elements.

Similarly the Ricci curvature tensor is a vector-valued $1$-form
which is given by inserting an inverse solder form (a $1$-form valued
vector) into the bivector-valued curvature $2$-form;
\begin{align}
  {\bf Rc}= \iota_{d{\cal P}^{-1}} {\bf Riem} = {\bf Riem}(\sovForm{\mu}\TSVec{\mu}) = \sovVec{\nu}
  R^{\mu\nu}_{\ \ \ \mu \tau} \TSForm{\tau}.
\end{align}
The inverse solder form induces a trace on the Riemann curvature
tensor.  Inserting the inverse solder form into the Ricci curvature
yields the scalar curvature,
\begin{align}
  R = \iota_{d{\cal P}^{-1}}\iota_{d{\cal P}^{-1}} {\bf Riem}= {\bf Riem}(\sovForm{\mu}\TSVec{\mu}, \sovForm{\nu}\TSVec{\nu}) =
  {\bf Rc}(\sovForm{\nu}\TSVec{\nu})= R^{\mu\nu}_{\ \ \ \mu \nu} .
\end{align}

In this section we will illustrate some properties of the
representation of curvature via local operators in the PF manifold
using the inherent discrete structure.  Our general strategy is to
formalize the notion of parallel transport and curvature interior to
the hybrid cells at the scale lengths shorter than the local scale of
the discretization.  We then generate a representation of the Riemann
curvature operator from the scale of the discretization, i.e. a zeroth
order coarse-graining over the interior structure of the hybrid cells.
This process is then used to reexamine the Ricci tensor from
\cite{McDonald:RicciRC}.

\subsection{The Riemann Tensor, Locally Einsteinian Structure, and the Conic Singularity}\label{subsec:EinsConic}

Given the PF manifold $\PFtri$, the domain of support associated to a
codimension $2$ hinge, $\hinge$, is the convex hull (interior to
$\PFtri$) of the hinge and a polygonal loop of parallel
transport. Meanwhile the meaningful domain is the measure which is
given by the oriented sum of irreducible hybrid cells.  In an
irreducible domain, there exists an orthogonal basis of vectors $\{
m_{i}\}$ which are defined the same as before.  Since the subspace of
the hinge $\hinge$ is isomorphic to $\mathds{R}^{d-2}$, we can choose a
complete set of the $\{ m_i\}$ for $0\leq i\leq d-3$ from any
irreducible hybrid cell containing $\hinge$.  The last two vectors of
any irreducible hybrid cell containing $\hinge$ then span $\hinge^*$.
We take these two vectors as $\MomArm{\hinge}{\ld}$ and $\ld$.  In the
full measure of $V_{\hinge}$, we have a set of $\{\MomArm{\hinge}{\ld},
\ld\}$ for each $\ld \in \latbound\hinge^{*}$ such that $\hinge^{*} =
\frac{1}{2} \sum_{\ld \in \latbound{\hinge^{*}}} \ld\wedge
\MomArm{\hinge}{\ld}$ (with a change in orientation as necessary to
ensure consistent orientation across the domain).  We instead focus on
a basis of bivectors (and $2$-forms) that leads to the representation
of ${\bf Riem}$ as a ${d \choose 2} \times {d\choose 2}$ matrix.  We
then ask ``How do we assign coefficients to the discretized Riemann
tensor in this basis?''

On the interior of the hybrid cell $V_{\hinge}$ the tangential
components of the metric on the boundary between two neighboring
tangent spaces are constant while there is generally a discontinuous
change in the normal components of the metric across the boundary,
when viewed in the basis of one of the $d$-simplexes \cite{FL:1984,
  Williams:GeoDev1}.  As we encircle a hinge with a loop of parallel
transport, we notice that the components of the metric tangential to
$\hinge$ are always constant across this domain (the surfaces parallel
to $\hinge$ are always flat and remain parallel to $\hinge$).  Hence,
the components of any vector in the space spanned by $\hinge$ will be
unaffected by parallel transport around $\hinge$.  In our vector
basis, the $\{ m_{i}\}$ for $0\leq i\leq d-3 $ will be unaffected by
parallel transport.  However, the $\MomArm{\hinge}{\ld}s$, which is
tangential to a given $\polys{d-1}$, will generally have components
normal to any other $\polys{d-1} \in V_{\hinge}$ and so will
experience a rotation when transported around a loop encircling
$\hinge$.  The amount of this rotation is always given by the deficit,
$\defect{\hinge}$.  Moreover, this is true regardless of the size of
the loop and depends only on whether the loop has a non-trivial
projection into the plane orthogonal to $\hinge$, i.e. $\hinge^*$.
This indicates that the basis we have chosen is anomalous and that for
a more robust analysis we must take care in our choice of basis.

\begin{thm}\label{thm:Riem}
The Riemann tensor, $\bf{Riem}$, on codimension $2$ hinges of $\PFtri$ are rank $1$ tensors with an eigen-decomposition
$$ {\bf Riem} = \hat{\hinge}^{*}\ {d \choose 2} \frac{\defect{h}}{A_{h}^{*}} \ \hinge^*.$$
\end{thm}
\begin{proof}
\begin{subequations}
  We assign a basis $\{ \sigma^{i}\}$ such that for each $\sigma^{i}$
  we have
\begin{equation}\label{eq:BasisGenPos}
\sigma^{i} \cdot \hinge^* = 2g^{\mu\nu} g^{\alpha \beta}
\sigma^{i}_{(\mu\alpha)} \hinge^{*}_{\nu\beta} \neq 0
\end{equation}
where $\sigma^{i}_{(\mu\alpha)} =
\frac{1}{2}\left[\sigma^{i}_{\mu\alpha} +
  \sigma^{i}_{\alpha\mu}\right]$.  Since each $\sigma^i$ has a
non-zero projection onto $\hinge^{*}$, we have
\begin{align}
{\bf Riem}(\sigma_{i})  = \hat{\hinge}^{*} \defect{\hinge}.
\end{align}
Therefore the Riemann tensor associates to each basis bivector
$\sigma_{i}$ a rotation bivector oriented along $\hinge^*$ with
magnitude of rotation equal to $\defect{\hinge}$.  Here we have
inserted an oriented area in ${\bf Riem}$ to obtain
\begin{align}
{\bf Riem}(\sigma_{i})  = \hat{\hinge}^{*} \defect{\hinge}.
\end{align}
This allows us to assign a matrix representation to the Riemann
curvature tensor,
\begin{align}\label{eq:RiemRCAsymm}
{\bf Riem} \doteq \left(      \begin{array}{cccc}
\defect{\hinge}\  (\hat{\hinge}^{*}\cdot \sigma^{1}) & \defect{\hinge}\  (\hat{\hinge}^{*} \cdot \sigma^{1})  
& \defect{\hinge} (\hat{\hinge}^{*} \cdot \sigma^{1})& \cdots \\
\defect{\hinge}\  (\hat{\hinge}^{*} \cdot \sigma^{2})  & \defect{\hinge}\  (\hat{\hinge}^{*} \cdot \sigma^{2}) 
& \defect{\hinge} (\hat{\hinge}^{*} \cdot \sigma^{2})  &\cdots \\
\defect{\hinge}\  (\hat{\hinge}^{*} \cdot \sigma^{3})  & \defect{\hinge}\  (\hat{\hinge}^{*} \cdot \sigma^{3}) 
& \defect{\hinge} (\hat{\hinge}^{*} \cdot \sigma^{3})  &\cdots\\
\vdots & \vdots& \vdots & \ddots
\end{array}\right),
\end{align}
where each row has identical elements and the sum of each column gives
$\defect{\hinge} \hat{\hinge}^{*}$.  Since a Riemann tensor is a
symmetric tensor across the basis of 2-forms, we require that the
matrix representation adhere to the symmetry.  The asymmetry that
appears in \refEQ{eq:RiemRCAsymm} is due to our asymmetric use of the
conic singularity.  To account for this, we further require that the
inner product $\hat{\hinge}\cdot\sigma^{i} $ be normalized in both the
space of values and tangent space.  This is equivalent to requiring
that any loop of parallel transport is treated as a 2-surface in the
space with constant sectional curvature equal to the sectional
curvature along $\hinge^{*}$.  Hence any loop with non-trivial
projection onto $\hinge^{*}$ gets maximally projected onto the dual
polygon $\hinge^{*}$ as in \refFig{fig:hinproj} and the Riemann tensor
acts only on this projected loop.
\begin{center}
\begin{figure}[h]
\includegraphics[width=2in]{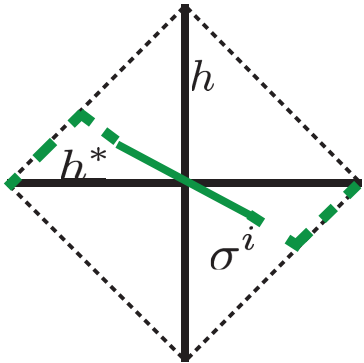}
\caption{The conic singularity has the peculiar property that the
  Riemann curvature tensor operates on any area of parallel transport
  (with non-zero projection onto $\hinge^*$) as if the area were
  projected identically onto $\hinge^{*}$.  While we take the area of
  this projection to be that of $\hinge^{*}$, this is not a necessity
  of the PF manifold.  All that is required is that each area only be
  acted upon by the sectional curvature of the plane $\hinge^*$.
  Taking the area to be equal to $A_{\hinge^*}= |\hinge^*|$ is a normalization
  choice that is natural given the dual operation defined and the
  inherent orthogonality of the dual lattices, see for example
  \cite{McDonald:DualTess, McDonald:RF}. }\label{fig:hinproj}
\end{figure}
\end{center}
The Riemann tensor becomes a matrix with uniform entries, i.e. a rank
$1$ matrix with eigenvalue given by the sum over a row or column.
Since there are ${d \choose 2}$ basis $2$-forms, we have a Riemann
tensor with a single component
\begin{equation} 
{\bf Riem} = \hat{\hinge}^{*} \frac{d(d-1)}{2} \
  \defect{\hinge} \ \hat{\hinge}^* \doteq
  \frac{d(d-1)}{2}\left( \begin{array}{cccc}
      \defect{\hinge} & 0 & 0 & \cdots \\
      0   & 0  & 0   &\cdots \\
      \vdots & \vdots& \vdots & \ddots
\end{array}\right).
\end{equation}
If we then put the basis of the horizontal section as the full lattice
elements $\hinge^{*}$ and the orthogonal $2$-forms, then we have
\begin{equation} {\bf Riem} = \hat{\hinge}^{*} \frac{d(d-1)}{2} \
  \frac{\defect{\hinge}}{A_{h^*}} \ {\hinge}^*.
\end{equation}
\end{subequations}
\end{proof}

In the eigenvalue analysis, we have assigned a projection operation on
the hybrid cell that treats any non-trivial loop of parallel transport
as a parallel transport around $\hinge^*$.  This tells us that the
Riemann tensor in this domain has the special property that there is
exactly one non-zero sectional curvature in the plane $\hinge^{*}$,
and a trivial flat subspace of bivectors orthogonal to $\hinge^*$.

While this is now a local measure of the Riemann tensor as obtained
from an arbitrary sampling of the space and evaluation of its
eigenspace decomposition, it does not carry with it the flavor of the
Riemann tensor as one would see from an evaluation on a single loop of
parallel transport.  This comes from the imposition of the conic
singularity that all bivectors (with non-zero projection on
$\hinge^*$) have assigned to them a equal sectional curvatures.
Therefore, a basis in general position (all basis elements satisfying
\refEQ{eq:BasisGenPos}) acts as though the space were an Einstein
space.  The eigenvalue is the local measure over the hybrid cell and
is the value through which we perform analysis in the DEC.

From this eigen-decomposition of the Riemann tensor we can regain a
local representation by normalizing the non-trivial eigenvalue by the
combinatoric factor counting the basis elements of a basis in general
position. In our case there are $ {d \choose 2}$ $2$-forms and the
normalization factor is given by ${d \choose 2}^{-1} = 1/{d \choose
  2}$.  This allows us to back-track from the eigenvalue ( as a
measure of all surfaces in the domain) to the local evaluation of the
Riemann tensor as a measure of the $2$-surfaces in the domain of
support;
$${\bf Riem} = \hat{\hinge}^{*}\ \ \frac{\defect{\hinge}}{A_{h^{*}}} \ \  {\hinge}^{*}. $$
This normalization process is a final step in the description of the
discretized tensors on the PF manifold.  It should be noted that any
and all calculus is done on the local measures of the differential
forms and tensors, i.e. the tensors integrated over the local volume.
The normalized, local tensors are convenient representations that are
made possible by the locally simple structure of the hybrid cells.

Further, we can assign a Ricci tensor and Ricci scalar in this hybrid
domain.  If we span $h^{*}$ by the local choice of
$\{\MomArm{\hinge}{\ld},\ld\}$, then we have
\begin{subequations} 
\begin{equation}\label{eq:UnNormRiem}
\bar{R}^{\MomArm{\hinge}{\ld} \ld}_{\ \ \ \ \ \ \MomArm{\hinge}{\ld} \ld} = -\bar{R}^{\ld \MomArm{\hinge}{\ld}}_{\ \ \ \ \   \MomArm{\hinge}{\ld} \ld} = -\bar{R}^{\MomArm{\hinge}{\ld} \ld}_{\ \ \ \ \ \ \ld \MomArm{\hinge}{\ld}} = \bar{R}^{\ld\MomArm{\hinge}{\ld}}_{\ \ \ \ \ \ \ld \MomArm{\hinge}{\ld}} = {d \choose 2} \frac{\defect{\hinge}}{A_{\hinge^*}},
\end{equation}
which leads to Ricci tensor components
\begin{equation}\label{eq:UnNormRic}
\bar{R}^{\MomArm{\hinge}{\ld}}_{\ \ \ \ \  \MomArm{\hinge}{\ld}} = \bar{R}^{\ld}_{\ \ \ld}  = {d \choose 2}\frac{\defect{\hinge}}{A_{\hinge^*}},
\end{equation}
and scalar curvature
\begin{equation}\label{eq:UnNormScal}
\bar{R}= 2 {d \choose 2}\frac{\defect{\hinge}}{A_{\hinge^*}},
\end{equation}
\end{subequations}
where the barred notation indicates the unnormalized representation of
the eigenvalues.  Normalizing these curvature tensors given their
differential forms character requires one to normalize by a factor of
${d\choose 2}$ for the Riemann tensor, ${d \choose 1}$ for the Ricci
tensor and ${d \choose 0}$ for the Ricci scalar.  Hence the normalized
curvature tensors are
\begin{subequations}
\begin{equation}\label{eq:NormRiem}
{R}^{\MomArm{\hinge}{\ld} \ld}_{\ \ \ \ \ \ \MomArm{\hinge}{\ld} \ld} = -{R}^{\ld \MomArm{\hinge}{\ld}}_{\ \ \ \ \   \MomArm{\hinge}{\ld} \ld} = -{R}^{\MomArm{\hinge}{\ld} \ld}_{\ \ \ \ \ \ \ld \MomArm{\hinge}{\ld}} = {R}^{\ld\MomArm{\hinge}{\ld}}_{\ \ \ \ \ \ \ld \MomArm{\hinge}{\ld}}  = \frac{\defect{\hinge}}{A_{\hinge^*}}
\end{equation}
\begin{equation}\label{eq:NormRic}
{R}^{\MomArm{\hinge}{\ld}}_{\ \ \ \ \  \MomArm{\hinge}{\ld}} = {R}^{\ld}_{\ \ \ld}=\frac{d-1}{2}\frac{\defect{\hinge}}{A_{\hinge^*}}
\end{equation}
\begin{equation}\label{eq:NormScal}
R = 2 {d \choose 2}\frac{\defect{\hinge}}{A_{\hinge^*}}.
\end{equation}
\end{subequations}
These are the oriented and normalized versions of the Riemann, Ricci
and scalar curvatures within a hybrid domain $V_{\hinge}$.  We often
only need the unoriented measure of the curvature forms and so summing
over the orientations introduces a factor of $2$ into
Eqs. (\ref{eq:UnNormRiem}), (\ref{eq:UnNormRic}), (\ref{eq:NormRiem})
and (\ref{eq:NormRic}).  We now have an accounting of the measures of
curvature on the hybrid domain on a hinge that is analogous to that
obtained by Friedberg and Lee \cite{FL:1984}.  On a qualitative scale,
we have the same form of the Riemann, Ricci and scalar curvatures
obtained in \cite{FL:1984}.  The quantitative distinction comes from
our use of the hybrid volume as a domain of support and our treatment
of the Dirac distribution on $\hinge$ as spread out over the entire
domain $V_{\hinge}$, instead of distributionally valued only on
$\hinge$.

\subsection{The Ricci tensor and its double dual}\label{subsec:RicciDD}

The Riemann tensor had a natural association to the dual polygons
$\hinge^*$ to the codimension $2$ hinges $\hinge$ given that its
differential form properties are that of a bivector-valued $2$-form.
Hence one need only project two indices of the Riemann tensor onto the
discretization.  In evaluating the Riemann tensor in the simplicial
discretization we also were able to express representations of the
Ricci tensor and scalar curvature in this domain.  We now shift
attention to the PF representation of the Ricci tensor.  Since the
Ricci tensor's natural representation is that of a vector-valued
one-form,  its direct discretization is on the 1-skeletons of the
dual and simplicial lattices.  In \cite{McDonald:RicciRC} we derived
representations of the Ricci tensor in both the dual and simplicial
1-skeletons as weighted averages of the Riemann curvatures.  We now
try to elucidate the properties of these derivations and draw
comparisons with an updated understanding of the curvature.

The Ricci tensor is given by the bivector-valued $2$-form curvature
operator acting on the inverse solder form $d{\cal P}^{-1} =
\sovForm{\mu}\TSVec{\mu}$, inducing a trace on the Riemann tensor.  As
a vector-valued $1$-form, the Ricci tensor is directly discretized on
$1$-forms orthogonal to the hinges.  Since ${\bf Riem}$ only has
components in the planes $\hinge^{*}$, ${\bf Rc}$ only takes
components along $\ld$ or $\MomArm{\hinge}{\ld}$.  The components
along $\MomArm{\hinge}{\ld}$ trivially give only one component from
${\bf Riem}$ and are contained in $V_{\hinge}$.  Calculating a Ricci
tensor in the direction of $\MomArm{\ld}{\hinge}$ becomes a sum over
directions orthogonal to $\MomArm{\ld}{\hinge}$, only one of which
gives a non-zero contribution.  Moreover, the $\MomArm{\hinge}{\ld}$
are virtual--being members of neither the simplicial or dual
lattices-- carry no inherent discrete differential forms.  This is
simply a statement about the non-independence of the curvature
directed along $\MomArm{\hinge}{\ld}$ and the Riemann curvature
associated to $\hinge$.  At the same time, in a small domain
surrounding any given $\ld$, there are $d$ distinct holonomies with
independent curvature operators.  Each of these distinct curvature
tensors have Ricci curvature components oriented along $\ld$.
Therefore, there exist non-trivial representation of ${\bf Rc}$s, and
distinct from the ${\bf Riem}$ of the hinges, on the $\ld$'s of the
dual lattice.  In \cite{McDonald:RicciRC} we sought to ensure that the
integrated measure of curvature associated $\hinge^{*}$ and $\ld$ was
preserved over the domain common to these two elements.  This is
equivalent to the continuum requirement that
$$ \text{tr}\LIP{\bf Rc}{\TSForm{a}}_{\Omega} = \text{tr}\LIP{\bf
  Riem}{\TSForm{a} \wedge \TSForm{b}}_{\Omega}
$$ 
over some common domain $\Omega$.  The measure of curvature one
obtains from integration over the domain is given by the one
non-trivial eigenvalue of the curvature in that domain, and hence the
unnormalized measure.  

To trace the Riemann curvature in the domain of
$\ld$ is  to sum over the polygonal loops $\hinge^{*} \ni \ld$
given that the individual domains of overlap between $\ld$ and each
$\hinge^*$ satisfies
\begin{equation}
\bar{R}_{\ld} V_{\ld\hinge^{*}} = \bar{R}_{\hinge^*}V_{\hinge^* \ld}.
\end{equation}
Doing so gives a scalar measure of the Ricci curvature on $\ld$
\begin{equation}
\bar{R}_{\ld} = \frac{\sum_{\hinge^*} R_{\hinge^*} V_{\hinge^* \ld}}{V_{\ld}}.
\end{equation}
Given that this is an integrated measure, it samples all orientations
of the loops of parallel transport and one naturally picks up the
unnormalized, integrated form of the Riemann curvature tensor and an
overall factor $d(d-1)$, 
\begin{equation}\label{eq:UnNormRicLD}
\bar{R}_{\ld}  = d(d-1)\volav{\frac{\defect{\hinge}}{A_{h^*}}}{\ld},
\end{equation}
where we have used the volume-weighted average
$$\volav{A_{\hinge}}{\ld} = \frac{ \sum_{\hinge^{*}\ni\ld}
  A_{h}V_{\hinge \ld }}{V_{\ld}}.
$$
This is an association of a scalar quantity to a $1$-form on the
lattice.   We can again normalize by the dimension of the space of $1$-forms
to obtain
\begin{equation}\label{eq:NormRicLD}
\bar{R}_{\ld}  = (d-1)\volav{\frac{\defect{\hinge}}{A_{h^*}}}{\ld}.
\end{equation}

We can further associate to this measure a directionality.  As we have
already mentioned, each of the curvature tensors contributing to
$R_{\ld}$ ($\bar{R}_{\ld}$) is already oriented along $\ld$ in each
subdomain $V_{\ld\hinge}$ since the measure of ${\bf Rc}$ on those
domains is obtained by contraction of ${\bf Riem}$ with
$\MomArm{\ld}{\hinge}$.  Therefore, the measures $\bar{R}_{\ld}$ and
$R_{\ld}$ can be considered as coefficients on the one-form oriented
along $\ld$.  This assigns both directionality and magnitude to ${\bf
  Rc}$ on a given $\ld$.

It is useful to note at this point that we have associated a scalar
quantity to ${\bf Rc}$ on $\ld$ and assigned to it a directionality.
However, the scalar coefficient is dependent on the domain of
integration.  While we associate to $\ld$ a component of ${\bf Rc}$ in
the direction of $\ld$, this object is not of the same class as the
${\bf Riem}$ on the hinges.  Whereas the curvature operators on the
hinges are constant over the codimension 2 hinges and treated as
constant over the domains $V_{\hinge}$ (whenever a subdomain also
encompasses the hinge itself), the Ricci curvatures are composed
explictly in terms of components of curvature operators whose
valuations only make sense within the subdomains $V_{\ld}{\hinge}$.
It is therefore notable that the scalar coefficient is only an
appropriate measure whenever a domain of interest encompasses the
entire $V_{\ld}$.  If a domain of interest only intersects a portion
of $V_{\ld}$, then one must suitably restrict the measures in the
definition of $R_{\ld}$ ($\bar{R}_{\ld}$).  This will come into play
as we now seek representations of ${\bf Rc}$ on the simplicial
lattice.

The natural discretization of ${\bf Rc}$ is on the dual lattice; however,
we have shown in \cite{McDonald:RicciRC} that a representation on the
simplicial edges is also possible by requiring that domains of overlap
between $V_{\ld}$ and $V_{\ls}$ give rise to the same measure of
integrated curvature.  We now want to show that is related to the
discrete version of the $1$-form double-dual of a vector-space valued
$1$-form.

In each irreducible hybrid cell common to both $\ld$ and $\ls$, we can
form a basis from the two vectors $\ls$ and $\ld$ as well as the set
of $d-2$ vectors spanning the subspace $\MomArm{\ls}{\ld}$, $\{
\Solder{i} \ |\ 1\leq i \leq d-2\}$.  The transformation from the dual
lattice to the simplicial lattice is done using the orthogonal
subspace $\MomArm{\ls}{\ld}$ applied through the commutative diagram
from \refEQ{eq:DD2}.  If we take the discrete form along $\ld$, then
we have $(\bar{R}_{\ld})\hspace{1.5 px} {\ld}$ as a vector-valued $1$-form with $\ld$ as
the trivial map from $\ld$ in the horizontal section to $\ld$ in the
vertical section.  Taking the dual of this assigns $\bar{R}_{\ld}$ to
$\ld^{*}$.  If a given $\ls$ is not contained in $\ld^{*}$ there is no
contribution of $\bar{R}_{\ld}$ to $\bar{R}_{\ls}$ and likewise in the
reverse. So we only consider when $\ld \in \ls^{*}$ and $\ls \in
\ld^{*}$.  Moreover, $\bar{R}_{\ld}$ is constructed from the Riemann
tensors evaluated on $\hinge^{*}$s and so we further expand the ${\bar
  R}_{\ld}$'s such that the dependence on $\bar{R}_{\hinge^*}$ is
  explicit.    If we then take the double dual and inner-derivative
  over the inverse solder forms on $\MomArm{\ls}{\ld}$,  we obtain
\begin{align} \label{eq:RicciDD}
\bar{R}_{\ls} V_{\ls\ld}  = \LIP{\iota_{\MomArm{\ls}{\ld}} \left(
    \star\bar{R}_{\ld}\right)}{\ls}_{V_{\ls\ld}} = \LIP{
  \star\bar{R}_{\ld} }{\ls
  \wedge \MomArm{\ls}{\ld}}_{V_{\ls\ld} }= \LIP{\bar{R}_{\ld}}{
    \star\left(\ls \wedge \MomArm{\ls}{\ld} \right) }_{V_{\ls\ld}} =
  \bar{R}_{\ld} V_{\ls\ld}.
\end{align}
This object is a measure of the contribution from those curvature
tensors with support in the defined domain.  The above integrated
curvatures are entirely connected to their domains of integration and
the restriction of the domains induces a restriction of the
integration on their definitions.  We then view \refEQ{eq:RicciDD} as
statement of dependence of the simplicial lattice Ricci curvature on
the restriction of the dual skeleton Ricci curvature.  As the latter
depends on multiple curvature tensors in multiple domains, the
restriction ensures that only those curvature operators with values in
the specified domain contribute the final result. Summing over all
$\ld$ that are incident to $\ls$, i.e. all $\ld \in \ls^*$, gives the
final measure of the unnormalized Ricci tensor on $\ls$;
\begin{align}
\bar{R}_{\ls} = \frac{\sum_{\ld \in \ls^{*}}
 \bar{ R}_{\ld}V_{\ld\ls}}{V_{\ls}}  = \frac{\sum_{h \ni \ls} \bar{R}_{\hinge}
  V_{\hinge\ls} }{V_{\ls}} = d(d-1) \frac{ \areaav{\defect{\hinge}}{\ls}}{\areaav{A_{h^*}}{\ls}}
\end{align}
where 
$$ \areaav{B_{\hinge}}{\ls} :=  \frac{ \sum_{\hinge \ni \ls}
  B_{\hinge}A_{\hinge \ls}}{\sum_{\hinge \ni \ls}
  A_{\hinge\ls}}, 
$$  
and we have used the relation
$$ V_{\ls} =  \sum_{\hinge \ni \ls} \frac{1}{{d\choose 2}} A_{\hinge \ls} A_{\hinge}^{*}.$$
It is should be clear from \refEQ{eq:RicciDD} that the object being
assigned to an $\ls$ is not the component of ${\bf Rc}$ along $\ls$,
but rather a measure of ${\bf Rc}$ in the orthogonal complement to
$\ls$. Summing over all $\ld$'s in $\ls^{*}$ provides a complete
measure of the ${\bf Rc}$ in that orthogonal complement to $\ls$.

After normalizing by the combinatorial factor ${d\choose 1}$
we get the normalized Ricci tensor on the simplicial edge $\ls$,
\begin{align}
{R}_{\ls} = (d-1) \frac{ \areaav{\defect{\hinge}}{\ls}}{\areaav{A_{h^*}}{\ls}}.
\end{align}
This demonstrates that we can formalize the transformation between the
dual and simplicial lattices via the trace of the double dual.
Moreover it provides a geometrically clear picture of the
transformation via the moment arm or generalized solder form between
$\ls$ and $\ld$.

The double-dual and its traces provide a direct path for the raising
and lowering operators common in general relativity and differential
geometry.  These were applied in \cite{McDonald:RicciRC} as a method
to obtain a dualized view of the Ricci tensor on the simplicial
lattice.  What is a particularly important lesson to be drawn from
this is that the Ricci tensor associated with any $\ls$ is not a
measure of the Ricci tensor in the direction of $\ls$ but an
average of Ricci tensors in the $(d-1)$-subspace orthogonal to $\ls$.
This was a crucial understanding in \cite{McDonald:RF} since one must
ensure that metric components in Hamilton's Ricci flow \cite{Hamilton:RF} change in
proportion to those components of the Ricci tensor.  If instead one
were to simply take metric components along $\ls$, Hamilton's Ricci
flow would not be recovered.  Rather one would obtain an orthogonal
flow to that of Hamilton's.

We have done this explicitly for the Ricci curvature and similar
results hold for the scalar curvature.   In particular, taking the
trace of the Ricci curvature, or  the double trace of the Riemann, we
assign a scalar curvature to a vertex, $\vd$, of the dual lattice;
\begin{equation}
R_{\vd} = \frac{\sum_{\hinge^{*}\ni \vd} \bar{R}_{\hinge^*}V_{\hinge^{*}
    \vd}}{V_{\vd} }.
\end{equation}
Since this is the scalar curvature, there is no distinction between
the normalized and unnormalized curvatures and hence we drop the bars
at the very beginning.     In \cite{McDonald:scalarRC} we showed that
scalar curvature on a vertex $\vs$ of the simplicial lattice is given
by 
\begin{equation}
 R_{\vs} = d(d-1)
\frac{\areaav{\defect{\hinge}}{\vs}}{\areaav{A_{\hinge^*}}{\vs}}.
\end{equation} 
We have thus described an intrinsic geometric derivation of the curvature operators and curvature scalar without explicit reference a limiting smooth sequence of surfaces.   

\section{Discussion and Conclusions}\label{sec:conclusion}

We have shown in this manuscript a revised formulation of DEC that is
based on the volume measures of differential forms on local domains of
compact support.  This formulation makes explicit the assignment of a
discrete form to a family of surfaces in a volume local to a lattice
element.  The characterization of a discrete form to the family of
surfaces provides a transparent view of the operations on discrete
forms, such as the exterior (co)-derivative.  This makes the DEC
approach more directly amenable to use in RC.

We built the volume-based DEC from the irreducible domains of the
lattice--the {\em monads} of space--that form the most basic
structures of the PF manifold.  These irreducible domains are domains
of supports for arbitrary $k$-forms in any given tangent space.  From
these irreducible cells we identified generalized solder forms and
local solder forms to allow for transformations between the simplicial
and dual lattices.  Moreover, since solder forms provide a unit map
from the tangent space (horizontal section) to the space of values
(vertical section), a discretization of the solder form allows one to
work within a framework that covers both scalar-valued and vector-space
valued differential forms.

We have also shown how curvature operators can be given explicit
constructions as bivector-valued two-forms in the lattice by examining
the conic singularities around the hinges $\hinge$ of $\PFtri$.  It
was found that the PF Riemann curvature operators have an
eigenspectrum with only one non-trivial eigenspace, that aligned along
the plane orthogonal to the codimension $2$ hinges.  Noticing that the
${\bf Riem}$ operator takes a form analogous to a space with constant
sectional curvature for bases in general position, we define local,
normalized curvature operators that can then be compared to local
continuum quantities.  This provides flexibility to insist that not
only should the average integrated curvature compare to the integrated
curvature over a finite domain, but that the local curvature operators
compare to sectional curvatures in the continuum.  It was then shown
how the Ricci curvature can be formulated on the dual lattice and how
the simplicial representation is viewed not as the components of ${\bf
  Rc}$ along a simplicial edge $\ls$ but the components of the double
dual of ${\bf Rc}$.

This general framework and the volume-based DEC is a purely discrete
foundation for the analysis of the geometry of PF manifolds.  The
continuum is only used at the level of the discretization and for the
locally flat behavior of irreducible hybrid cells.  We can therefore
characterize this approach as a stratified view of discrete geometry
with three distinct regimes: (1) the scale smaller than the local
discretization where one admits ignorance of the internal structure
and must assume some approximate behavior (e.g. flatness, constant
curvature, etc), (2) the discrete scale where the geometric properties
are based on the PF structure, and (3) a coarse-grained scale that
regains the continuum behavior of the manifold.  We have shown how to
treat the discrete scale by inferring behavior from the connectivity
of irreducible hybrid cells within a given domain.

These results have been applied in \cite{McDonald:RF} to a simplicial
discretization of Hamilton's Ricci flow and in earlier stages in
\cite{McDonald:RCMatter}.  These results utilize the foundations laid
in RC to open up DEC to a variety of geometric objects characterizable
as vector-space valued differential forms.

\begin{acknowledgments}
  WAM, XDG, ST Yau, acknowledge support from Air Force Research
  Laboratory Information Directorate (AFRL/RI) grant
  \#FA8750-11-2-0089 and through support from the Air Force Office of
  Scientific Research through the American Society of Engineering
  Education's Summer Faculty Fellowship Program at AFRL/RI's Rome
  Research Site.  JRM acknowledges support from a National Research
  Council Research Associateship at AFRL/RI's Rome Research Site.  We
  also wish to thank Shannon Ray, Chris Tison, Arkady Kheyfets and
  Matthew Corne for helpful discussions and suggestions on this
  manuscript.  Any opinions, findings and conclusions or
  recommendations expressed in this material are those of the
  author(s) and do not necessarily reflect the views of AFRL.
\end{acknowledgments}

\appendix

\bibliography{DRF}{  }
\bibliographystyle{apsrev4-1}

\end{document}